\documentclass[letterpaper, 11pt,  reqno]{amsart}

\usepackage{amsmath,amssymb,amscd,amsthm,amsxtra, esint}
\usepackage{tikz} 
\usepackage{color}
\usepackage[implicit=true]{hyperref}

\usepackage{cases}

\allowdisplaybreaks[2]

\usetikzlibrary{shapes.misc}
\usetikzlibrary{shapes.symbols}
\usetikzlibrary{shapes.geometric}

\tikzset{
	ddot/.style={circle,fill=white,draw=black,inner sep=0pt,minimum size=0.8mm},
	>=stealth,
	}

\tikzset{
	ddot2/.style={circle,fill=black,draw=black,inner sep=0pt,minimum size=0.8mm},
	>=stealth,
	}

\newtheorem{theorem}{Theorem} [section]

\newtheorem{lemma}[theorem]{Lemma}
\newtheorem{proposition}[theorem]{Proposition}
\newtheorem{remark}[theorem]{Remark}

\newcommand{\I}{\mathcal{I}}

\newcommand{\noi}{\noindent}
\newcommand{\Z}{\mathbb{Z}}
\newcommand{\R}{\mathbb{R}}
\newcommand{\C}{\mathbb{C}}
\newcommand{\T}{\mathbb{T}}

\let\Re=\undefined\DeclareMathOperator*{\Re}{Re}
\let\Im=\undefined\DeclareMathOperator*{\Im}{Im}

\let\P= \undefined
\newcommand{\P}{\mathbf{P}}

\newcommand{\E}{\mathbb{E}}

\newcommand{\F}{\mathcal{F}}

\newcommand{\al}{\alpha}
\newcommand{\be}{\beta}
\newcommand{\dl}{\delta}

\newcommand{\nb}{\nabla}

\newcommand{\Dl}{\Delta}
\newcommand{\eps}{\varepsilon}

\newcommand{\g}{\gamma}
\newcommand{\G}{\Gamma}
\newcommand{\ld}{\lambda}
\newcommand{\Ld}{\Lambda}
\newcommand{\s}{\sigma}

\newcommand{\ft}{\widehat}

\newcommand{\wt}{\widetilde}
\newcommand{\cj}{\overline}
\newcommand{\dx}{\partial_x}
\newcommand{\dt}{\partial_t}

\renewcommand{\l}{\ell}
\renewcommand{\o}{\omega}
\renewcommand{\O}{\Omega}

\newcommand{\les}{\lesssim}

\newcommand{\jb}[1]
{\langle #1 \rangle}

\newcommand{\ind}{\mathbf 1}

\newcommand{\N}{\mathbb{N}}

\renewcommand{\H}{\mathcal{H}}

\newtheorem*{ackno}{Acknowledgements}

\numberwithin{equation}{section}
\numberwithin{theorem}{section}

\begin{document}
\baselineskip = 15pt

\title[Renormalization of 2-$d$ stochastic NLW]
{Renormalization of the two-dimensional stochastic nonlinear wave equations}

\author[M.~Gubinelli, H.~Koch, and T.~Oh]
{Massimiliano Gubinelli, Herbert Koch, and Tadahiro Oh}

\address{
Massimiliano Gubinelli\\
Hausdorff Center for Mathematics \&  Institut f\"ur Angewandte Mathematik\\
 Universit\"at Bonn\\
Endenicher Allee 60\\
D-53115 Bonn\\
Germany}
\email{gubinelli@iam.uni-bonn.de}

\address{Herbert Koch\\
Mathematisches Institut\\
 Universit\"at Bonn\\
Endenicher Allee 60\\
D-53115 Bonn\\
Germany
}

\email{koch@math.uni-bonn.de}

\address{
Tadahiro Oh, School of Mathematics\\
The University of Edinburgh\\
and The Maxwell Institute for the Mathematical Sciences\\
James Clerk Maxwell Building\\
The King's Buildings\\
Peter Guthrie Tait Road\\
Edinburgh\\ 
EH9 3FD\\
 United Kingdom}

\email{hiro.oh@ed.ac.uk}

\subjclass[2010]{35L71, 60H15}

\keywords{stochastic nonlinear wave equation; nonlinear wave equation; 
renormalization; Wick ordering;
Hermite polynomial; white noise}

\begin{abstract}
We study the two-dimensional stochastic nonlinear wave equations (SNLW) with 
an additive space-time white noise forcing.
In particular, we introduce a time-dependent renormalization
and prove that SNLW is pathwise locally well-posed.
As an application of the local well-posedness argument, 
we also establish  a weak universality result 
for the renormalized SNLW.

\end{abstract}

\maketitle

\baselineskip = 14pt

\section{Introduction}

\subsection{Stochastic nonlinear wave equations}

We consider 
the following stochastic nonlinear wave equations (SNLW)  on $\T^2 = (\R/\Z)^2$
with an additive space-time white noise forcing:
\begin{align}
\begin{cases}
\dt^2 u -  \Dl  u   \pm  u^k = \xi\\
(u, \dt u) |_{t = 0} = (\phi_0, \phi_1) \in \H^s(\T^2) : = H^s(\T^2)\times H^{s-1}(\T^2), 
\end{cases}
\quad (x, t) \in \T^2\times \R_+,
\label{SNLW1}
\end{align}

\noi
where $k\geq 2$ is an integer and 
$\xi(x, t)$ denotes a (Gaussian) space-time white noise on $\T^2\times \R_+$.
In view of the time reversibility of the deterministic nonlinear wave equations, 
one can also consider \eqref{SNLW1} on $\T^2 \times \R$
by extending 
the white noise $\xi$ onto $\T^2 \times \R$.\footnote{Namely, 
replace
$\be_n$ in \eqref{Wiener1} by the sum of  two independent Brownian motions, one forward in time on $\T^2\times [0, \infty)$ and 
the other backward in time $\T^2 \times (-\infty, 0]$,
both starting at $t = 0$.}
For simplicity, however, we only consider positive times in the following.
Moreover, we restrict our discussion to the real-valued setting.

The stochastic wave equations with space-time white noise and with general nonlinearity 
have already been considered by Albeverio, Haba, Oberguggenberger, and Russo in a series of works~\cite{russo1,russo2,russo3,russo4} 
for spatial dimensions going from one to three. 
In particular, they showed that, in two and three dimensions, solutions have to be distributions.
Moreover, they highlighted a phenomenon of \emph{triviality};
let $u_\eps$ be a smooth solution of SNLW obtained by replacing the space-time white noise $\xi$
in \eqref{SNLW1}
by a suitable regularized noise  $\xi_\eps$.
Then, it was shown that as the regularization is removed,
$u_\eps$ converges  to a limiting process $u$
satisfying  a \emph{linear} wave equation. 
The nonlinear behavior does not appear any more in the limiting equation due to the extreme oscillations of prelimit solutions $u_\eps$. This phenomenon has been already noticed in parabolic equations, for example in the stochastic quantization problem of Euclidean scalar fields in two and three 
dimensions,
where 
 a renormalization is needed in order to obtain a non-trivial limiting behavior. 
 In this paper,  we will present the first analysis of the renormalization problem for the stochastic nonlinear wave equation \eqref{SNLW1} in two dimensions. 
 In order to  implement this renormalization at the algebraic level,
  we restrict   the form of the nonlinearity to a polynomial one. For the sake of simplicity, we consider a single monomial,  although 
 more general polynomial interactions could be considered. Other possible models for nonlinearity which should be amenable to renormalization are those given by trigonometric or exponential functions.
 In this case, however,  we expect the renormalization problem to be more subtle and 
 thus we leave it aside for the moment.

By letting $v = \dt u$, we can write \eqref{SNLW1}
in the following Ito formulation:
\begin{align}
\begin{cases}
d  \begin{pmatrix}
u \\ v
\end{pmatrix}
+ \Bigg\{
\begin{pmatrix}
0 & -1\\
-\Dl & 0
\end{pmatrix}
 \begin{pmatrix}
u \\ v 
\end{pmatrix}
+  
\begin{pmatrix}
0 \\ \pm u^k
\end{pmatrix}
\Bigg\} dt 
 = d \begin{pmatrix}
0  \\W
\end{pmatrix}\\
\rule{0mm}{6mm}(u, v) |_{t = 0} = (\phi_0, \phi_1).
\end{cases}
\label{SNLW2}
\end{align}

\noi
Here, $W$ denotes
a cylindrical Wiener process on $L^2(\T^2)$.
More precisely, by letting 
\[e_n(x) = e^{2\pi i n \cdot x}, \quad 
\I = (\Z_+\times \{0\}) \cup (\Z\times \Z_+),
\quad \text{and}\quad \mathcal{J} = \I \cup\{(0, 0)\},\]

\noi
we have\footnote{Note that $\{1, \sqrt 2\cos(2\pi n\cdot x), \sqrt 2 \sin (2\pi n \cdot x):
n \in \I \}$ forms an orthonormal basis of $L^2(\T^2)$ in the real-valued setting.}
\begin{align}
 W(t) & = \be_0(t) e_0 + \frac{1}{\sqrt 2} \sum_{n\in \Z^{2}\setminus\{0\}}
\be_n (t) e_n\notag \\
& = \be_0(t) e_0 +  \sum_{n\in \I}
\Big[\Re(\be_n (t)) \cdot \sqrt 2 \cos (2\pi n \cdot x)
- \Im(\be_n (t) )\cdot \sqrt 2 \sin (2\pi n \cdot x)\Big],
\label{Wiener1}
\end{align}

\noi
where $\{ \be_n\}_{n \in \mathcal{J}}$ is a family of  mutually independent complex-valued Brownian
motions\footnote{Here, we take $\be_0$ to be real-valued.}
on a fixed probability space $(\O, \F, P)$
and $\be_{-n} := \cj{\be_n}$ for $n \in \mathcal{J}$.
Note that $\text{Var}(\be_n(t)) = 2t$ for $n \in \Z^2 \setminus\{0\}$, 
while 
 $\text{Var}(\be_0(t)) = t$.
It is easy to see that $W$ almost surely lies in 
$C^{\al}(\R_+;H^{-1-\eps}(\T^2))$
for any $\al < \frac 12$ and $\eps > 0$.

Let $S(t)$ be the propagator for the linear wave equation
defined by 
\begin{equation*}
S(t)(\phi_0, \phi_1):=\cos(t|\nabla|)\phi_0+\frac{\sin (t|\nabla|)}{|\nabla|}\phi_1
\end{equation*}

\noi
as a Fourier multiplier operator.
Then, the mild formulation 
of the Cauchy problem \eqref{SNLW1} (and \eqref{SNLW2})
is given by 
\begin{equation*}
u(t)=S(t)(\phi_0, \phi_1)
\mp \int_{0}^t \frac{\sin ((t-\tau)|\nabla|)}{|\nabla|}u^k(\tau)d\tau
+ \int_{0}^t \frac{\sin ((t-\tau)|\nabla|)}{|\nabla|}dW(\tau).
\end{equation*}
	
\noi
In fact, as it is written, this problem is ill-posed since solutions are expected to be merely distributions in the space variable, raising the problem of controlling the nonlinear term. The problem is already apparent at the level of  the stochastic convolution: 
\begin{equation}
\Psi(t) \stackrel{\text{def}}{=} 
(\dt^2 - \Dl)^{-1} \xi
=  \int_{0}^t \frac{\sin ((t-\tau)|\nabla|)}{|\nabla|}dW(\tau).
\label{sconv1}
\end{equation}

\noi
It can be shown that for each $t> 0$, $\Psi (t)\notin L^2(\T^2)$ almost surely.
In particular, there is an issue in making sense of powers
$\Psi^k$ and a fortiori of the full nonlinearity $u^k$. 
As we discussed above, 
we need to modify  the equation in order to take into account a proper renormalization and a different nonlinearity has to be considered.

\subsection{Renormalized SNLW}
In order to explain the renormalization process,
 we first regularize the equation \eqref{SNLW1} by a Fourier truncation of the noise term and of initial data:\footnote{Strictly speaking, the regularization of initial data is not necessary here but it allows
 us to consider smooth solutions to the regularized equation \eqref{SNLW4}.}
\begin{align}
\begin{cases}
\dt^2 u_N -  \Dl  u_N   \pm  ( u_N)^k   = \P_N \xi \\ 
(u_N, \dt u_N) |_{t = 0} = (\P_N\phi_0, \P_N\phi_1),
\end{cases}
\label{SNLW4}
\end{align}

\noi
where  $\P_N$ is the Dirichlet projection onto the spatial frequencies $\Z^2_N \stackrel{\text{def}}{=} \{|n|\leq N\}$.
In the following, we discuss the renormalization for   \eqref{SNLW4}.

We define 
 the truncated stochastic convolution $\Psi_N(t)$ by 
\begin{align} 
\Psi_N(t) \stackrel{\text{def}}{=}  \P_N \Psi(x, t)
 = \sum_{n \in \mathbb{Z}^2_N} e_n 
   \int_0^t \frac{\sin ((t - \tau) | n |)}{| n |} d \wt \beta_n (\tau) 
\label{G0}
\end{align}

\noi
with the understanding that
\begin{align}
\frac{\sin ((t - \tau) | 0 |)}{| 0 |} \stackrel{\text{def}}{=} t-\tau
\qquad \text{and}\qquad
 \wt \be_n = \begin{cases}
\frac{1}{\sqrt{2}}\be_n, & \text{if }n \ne 0, \\
\be_0, & \text{if } n = 0.
\end{cases}
\label{G0a}
\end{align}

\noi
Then, 
for each fixed $x \in \T^2$ and $t \geq 0$,  
 it follows from  Ito isometry
that the random variable $\Psi_N(x, t)$
 is a mean-zero real-valued Gaussian random variable with variance
\begin{align}
\s_N(t) &  \stackrel{\text{def}}{=} \E \big[\Psi_N^2(x, t)\big]
 = \int_0^t (t-\tau)^2\, d\tau
+ 2 \sum_{n \in \I \cap \Z^2_N}
\int_0^t \bigg[\frac{\sin((t - \tau)|n|)}{|n|} \bigg]^2 d\tau
\notag \\
& = \frac{t^3}{3} +  \sum_{0< |n|\leq N} \bigg\{ \frac{t}{2 |n|^2} - \frac{\sin (2 t |n|)}{4|n|^3 }\bigg\}\sim  t \log N.
\label{sig1}
\end{align}

\noi
Note that $\s_N(t)$ is independent of $x \in \T^2$.
The structure of the equation makes it clear that any solution can be decomposed as 
\begin{equation}
\label{eq:ansatz}
u_N = \Psi_N + v_N, 
\end{equation}

\noi
where the residual term $v_N$ solves a nonlinear wave equation (NLW) with 
the following  polynomial nonlinearity 
with random coefficients depending on $\Psi_N$:
\begin{equation}
\label{eq:decomposition}
u_N^k = \sum_{\ell=0}^k {k\choose \ell} \Psi_N^\ell v_N^{k-\ell}.
\end{equation}

\noi
Note, however, that
  the monomials $\Psi_N^\ell$ does not have nice limiting behavior as $N\to \infty$. 
  Despite this difficulty,  the decomposition~\eqref{eq:ansatz} is motivated by the heuristics that, in two dimensions, the only singularities which have to be dealt with in the renormalization process are related to the powers of the random field $\Psi$. We are going to prove that this is indeed the case and that the residual term $v_N$ can be controlled in a nice space. The decomposition~\eqref{eq:ansatz} usually takes the name of Da Prato-Debussche trick \cite{DPD}
in the field of stochastic parabolic  PDEs.
Note that such an idea also appears in 
McKean \cite{McKean} and 
Bourgain \cite{BO96} in the context of (deterministic) dispersive PDEs with random initial data, predating \cite{DPD}. See also Burq-Tzvetkov \cite{BT1}.
 
In order to renormalize the nonlinearity $u_N^k$ in~\eqref{eq:decomposition},
 we need to introduce suitable counter-terms. We will show that in order to renormalize 
 each random monomial $\Psi^\ell_N$,
  it is enough to replace it with its \emph{Wick ordered} counterpart: 
\begin{align}
:\!\Psi^\l_N(x, t) \!: \, \stackrel{\text{def}}{=} H_\ell(\Psi_N(x, t);\sigma_N(t)).
\label{Herm1}
\end{align}

\noi
Here,  
$H_\l(x; \s )$ is given by 
\begin{align*}
H_\l(x; \s )  = \s^\frac{\l}{2} H_\l(\s^{-\frac{1}{2}} x),
\end{align*}

\noi
 where $H_\l(\cdot)$ is the $\l$th Hermite polynomial for the standard Gaussian measure.
Combining this with the following standard identity
\begin{align*}
 H_k(x+y) 
& = \sum_{\l = 0}^k
\begin{pmatrix}
k \\ \l
\end{pmatrix}
x^{k - \l} H_\l(y), 
\end{align*}

\noi
we have 
\begin{align}
H_k(x+y; \s )
&  = \s^\frac{k}{2}\sum_{\l = 0}^k
\begin{pmatrix}
k \\ \l
\end{pmatrix}
\s^{-\frac{k-\l}{2}} x^{k - \l} H_\l(\s^{-\frac{1}{2}} y) \notag \\
&  = 
\sum_{\l = 0}^k
\begin{pmatrix}
k \\ \l
\end{pmatrix}
 x^{k - \l} H_\l(y; \s).
\label{Herm3}
\end{align}

\noi
In our situation, this gives 
\begin{equation*}
H_k(u_N(x, t);\s_N(t)) = \sum_{\l=0}^k {k\choose \ell} H_\ell(\Psi_N(x, t);\sigma_N(t)) 
\big(v_N(x, t)\big)^{k-\ell}.
\end{equation*}

\noi
From this, we see that 
 Wick ordering all the monomials $\Psi_N^\ell$ in 
\eqref{eq:decomposition}
 is equivalent to replacing the original nonlinearity $u_N^k$ by the $k$th Hermite polynomial 
 $H_k(u_N(x, t);\sigma_N(t))$. 
Note that there is no reason for $u_N$ to be a Gaussian random variable.
By common abuse of language, however, we refer to the function $H_k(u_N(x, t);\sigma_N(t))$ 
as a Wick ordered non-linearity\footnote{We expect the variance of the solution $u_N(t)$
grows in time.  See Oh-Quastel-Sosoe \cite{OQS2} in the context of the stochastic KdV equation.
Hence, the renormalization must depend on time.
This is different from the situation where one expects an invariant measure
for a given dynamics so that a renormalization is time-independent.}
of $u_N^k$.
Compare this with the usual Wick ordered (deterministic) NLW
on $\T^2$ considered in Oh-Thomann \cite{OT2}.

As in the case of  the usual (time-independent) Wick ordered monomial, 
this time-dependent renormalization allows us
to define 
\begin{align}
:\! \Psi^k  \!: \,      \stackrel{\text{def}}= \lim_{N\to \infty}    :\! \Psi_N^k \!: 
\label{Wick2}
\end{align}

\noi
in $L^p(\O; C([0,T];W^{-\eps,\infty}(\T^2)))$
for any $p < \infty$ and $\eps >  0$
(and for any $k \in \N$).\footnote{Here, $W^{s, r}(\T^2)$
denotes the usual $L^r$-based Sobolev space (Bessel potential space) defined by the norm:
\[\| u\|_{W^{s, r}} = \| \jb{\nb}^s u \|_{L^r} = \big\| \F^{-1}( \jb{n}^s \ft u(n))\big\|_{L^r}.\]

\noi
When $r = 2$, we have $H^s(\T^2) = W^{s, 2}(\T^2)$.} 
See Proposition \ref{PROP:sconv} below. 
This convergence result allows us to describe the limiting problem we are going to solve. 
Consider a function $u=\Psi+v$,  where $v\in L^q([0,T];W^{s,r}(\T^2))$ for some appropriate $q,r \ge 1$ and $s >0$. 
Then, as $N \to \infty$, we have
\begin{align}
\label{eq:wick}
H_k(\P_N & u(x, t);  \sigma_N(t))  \notag\\
& \longrightarrow  \ :\!u^k(x, t)\!:\, = F_\Psi(v) (x, t)
\stackrel{\text{def}}{=} \sum_{\ell=0}^k {k\choose \ell} :\! \Psi^\l (x, t)  \!: \big(v(x, t)\big)^{k-\ell}.
\end{align}

\noi
We insist that the  nonlinear (random) function $u \mapsto \, :\!u^k\!:\,  =F_\Psi(v)$ 
is only defined for $u$ of the form $\Psi+v$ with suitable $v$. 
With this in mind,  we set our main goal to prove local well-posedness of 
the following Wick ordered SNLW:
\begin{align}
\begin{cases}
\dt^2 u -  \Dl  u   \, \pm : \!u^{k} \!:\, =  \xi \\
(u, \dt u) |_{t = 0} = (\phi_0, \phi_1). 
\end{cases}
\label{SNLW7}
\end{align}

\noi
In the following, we concentrate on 
the following mild formulation of  the  Wick ordered SNLW \eqref{SNLW7}:
\begin{align}
u(t)= & S(t)(\phi_0,    \phi_1)
 \mp \int_{0}^t  \frac{\sin ((t-\tau)|\nabla|)}{|\nabla|} \!:\! u^k(\tau)\!: \!d\tau 
 + \int_{0}^t \frac{\sin ((t-\tau)|\nabla|)}{|\nabla|}dW(\tau),
\label{SNLW8}
\end{align}

\noi
where the Wick ordered nonlinearity 
$\!:\! u^k\!:$ is defined by \eqref{eq:wick}.

We point out that such a solution $u$ to \eqref{SNLW7}
and \eqref{SNLW8}
can also be given as the limit of solutions to 
the following truncated Wick ordered SNLW:
\begin{align}
\begin{cases}
\dt^2 u_N -  \Dl  u_N   \pm H_k(u_N;\sigma_N)  = \P_N \xi \\ 
(u_N, \dt u_N) |_{t = 0} = (\P_N\phi_0, \P_N\phi_1) , 
\end{cases}
\label{SNLW5}
\end{align}
as $N \in \N$. 
More precisely, 
one can study 
the following mild formulation of  the truncated Wick ordered SNLW \eqref{SNLW5}:
\begin{align}
u_N(t)= & S(t)(\P_N\phi_0,   \P_N \phi_1)
 \mp \int_{0}^t  \frac{\sin ((t-\tau)|\nabla|)}{|\nabla|} H_k(u_N;\sigma_N)(\tau) d\tau+ \Psi_N(t) 
\label{SNLW6}
\end{align}

\noi
and prove  (i) \eqref{SNLW6} is locally well-posed
``uniformly in $N \in \N$''
and  (ii) 
 $u_N$ converges to 
a stochastic process $u$ such that 
the Wick ordered nonlinearity $:\! u^k (x, t) \!: \,$ in \eqref{SNLW7}
is well defined  and 
the following limit holds:
\begin{align*}
:\! u^k  (x, t)\!: \,    \stackrel{\text{def}}{=} \lim_{N\to \infty}    H_k(u_N(x, t);\sigma_N(t)).
\end{align*}

\noi
One can then define
 this limit $u$ to be a solution to \eqref{SNLW7}.
This solution $u$ constructed as a limit of $u_N$ as above
agrees with the solution to the mild formulation \eqref{SNLW8}
in a suitable sense.
See Remark \ref{REM:thm1} below.

\subsection{Main result}

Before we state our main result, we first need to discuss
critical regularities associated to  the deterministic NLW:
\begin{align*}
\dt^2 u -  \Dl  u   \pm  u^k = 0.
\end{align*}

\noi
On the one hand, NLW on $\R^d$ enjoys the scaling symmetry, 
which induces the so-called scaling critical Sobolev index:
$s_\text{scaling} = \frac{d}{2} - \frac{2}{k-1}$.
On the other hand, NLW also enjoys
the Lorentzian invariance (conformal symmetry), 
which yields its own critical regularity 
$s_\text{conf} = \frac{d+1}{4} - \frac{1}{k-1}$
(at least in the focusing case);
see \cite{LS}.
\noi
In particular, when $d = 2$, we define $s_\text{crit}$ for a given integer $k \ge 2$ by 
\begin{align}\label{scrit}
s_\text{crit} := \max(s_\text{scaling}, s_\text{conf}, 0) = \max\bigg(1 - \frac{2}{k-1}, 
\frac 34 - \frac1{k-1}, 0\bigg).
\end{align}

\noi
Note that the third regularity restriction
$0$
appears in making sense of powers of $u$.
See also \eqref{scrit2} and Figure \ref{FIG:1} below.

We now state our main result.

\begin{theorem}\label{THM:1}
Given  an integer  $k \geq 2$, let $s_\textup{crit}$ be as in \eqref{scrit}.
Then, the Wick ordered SNLW \eqref{SNLW7}
is pathwise locally well-posed in $\H^s(\T^2)$
for 
\[ \textup{(i) } k \geq 4: 
\ s \geq  s_\textup{crit}
\qquad \text{or} \qquad 
 \textup{(ii) }  k = 2, 3: 
\ s >  s_\textup{crit}.
\]

\noi
More precisely, 
given any $(\phi_0, \phi_1) \in \H^s(\T^2)$, 
there exists a stopping time $T =  T_\o( \phi_0, \phi_1)$
(which is positive almost surely)
such that there exists  a unique solution $u $ to the mild formulation \eqref{SNLW8} on $[0, T]$
with 
\begin{align*}
u \in    \Psi + C([0, T]; H^\s(\T^2)) 
\subset C([0, T]; H^{-\eps}(\T^2)) 
\end{align*}
	
\noi
for any $\eps > 0$, where $\s = \min(s, 1-\eps)$.

\end{theorem} 

In Theorem \ref{THM:1}, 
the uniqueness holds only in 
 $\Psi + X^\s(T)$, 
  where $X^\s(T)$ is given by 
\[ X^\s(T) = C([0,T];H^\s(\T^2))\cap C^1([0,T]; H^\s(\T^2))\cap L^q([0, T]; L^r(\T^2))
\]

\noi
for some suitable $\s$-admissible pair $(q, r)$.
See Section \ref{SEC:3} for more on this point.

In \cite{Oh}, the third author proved pathwise local well-posedness of  
the following stochastic KdV
 with an additive space-time white noise
forcing:
\[ d u + (\dx^3 u  + u \dx u)dt =dW, \qquad (x, t) \in \T \times \R_+, \]

\noi
where $W$ denotes
a cylindrical Wiener process on $L^2(\T)$.
Theorem \ref{THM:1} is the second example on pathwise local well-posedness 
of rough stochastic dispersive PDEs
 with an additive space-time white noise forcing.

As we already mentioned in the previous subsection, 
the Wick ordered SNLW is \emph{defined} only for functions 
\begin{align}
u = \Psi+v
\label{decomp}
\end{align}

\noi
 with $v$ of suitable positive regularity. 
The main strategy for proving  Theorem \ref{THM:1} 
is then to consider the following fixed point problem for $v = u  - \Psi$:
\begin{align}
v(t)= 
S(t) (\phi_0, \phi_1)
 \mp \int_{0}^t  \frac{\sin ((t-\tau)|\nabla|)}{|\nabla|}  F_\Psi(v(\tau)) d\tau,
\label{SNLW9}
\end{align}

\noi
where $F_\Psi$ is as in \eqref{eq:wick}.

The proof is based on a fixed point argument via the Strichartz estimates
for the wave equations and  the general structure of the proof is similar to that 
for stochastic parabolic equations. The key point is to use function spaces
where the wave equation allows for a gain in regularity. 
This gain is sufficient to prove that $v$ has better regularity than $\Psi$ and gives a
well defined nonlinearity for which suitable local-in-time estimates can be established.  
In Section~\ref{SEC:2},  we prove the necessary
stochastic estimates for the random terms and then we give the
deterministic nonlinear estimates and the proof of
Theorem~\ref{THM:1} in Section~\ref{SEC:3}.

As an application of the local well-posedness argument, 
we  show a weak universality result for the Wick ordered SNLW
in Section \ref{SEC:4}.
Given small $\eps > 0$, we consider the following  SNLW equation 
on a dilated torus $(\eps^{- 1} \T)^2$ 
with a smooth noise $\eta^\eps$:
\begin{align*}
\begin{cases}
 \dt^2 w_\eps - \Dl w_\eps = f (w_\eps) + a (\eps, t) w_\eps + \delta   (\eps) \eta^\eps \\
 (w_\eps, \dt w_\eps) |_{t = 0} = (0, 0),
\end{cases}
\qquad 
(x, t)\in (\eps^{- 1} \T)^2 \times \R_+, 
\end{align*}

\noi
where $f: \R \to \R$ is a given smooth odd, bounded function
with a sufficiently number of bounded derivatives,
 $\eta^\eps$ is a noise which is white in time but
smooth and stationary in space,
and $a (\eps, t)$ and $\delta(\eps)$ are parameters to be chosen.
Consider the following space-time scaling:
\begin{align*}
u_{\eps} (x, t) = \eps^{- \g} w_\eps (\eps^{-1} x, \eps^{-1} t)
\end{align*}

\noi
for some $\g > 0$.
Namely, $u_\eps$ describes the behavior of $w_\eps$ at large scales, 
both in space and time.
Then, by appropriately choosing 
 parameters  $\g = 1$,   $\delta(\eps) = \eps^{\frac 32}$, and
 $a (\eps, \eps^{-1} t)$, 
 we show that $u_\eps$ converges in a suitable sense to 
  the solution $u$ to the Wick ordered cubic SNLW:
  \begin{equation*}
\begin{cases}
\dt^2 u - \Dl  u = \lambda : \!u^3 \!: \, + \, \xi\\
 (u, \dt u) |_{t = 0} = (0, 0)
\end{cases}
  \end{equation*}

\noi
for some $\ld = \ld(f)$.
Here, we can choose  
 $a(\eps, t)$ such that it depends only on $f$, the noise, and $\eps > 0$.
See Theorem \ref{THM:2} below.
We also refer readers to 
\cite{HQ, GP1, GP2} for more discussion on weak universality 
(for stochastic parabolic equations, in particular the KPZ equation).

\medskip

We conclude this introduction by stating several remarks.

\begin{remark}\label{REM:thm1} \rm
The same local well-posedness result
also applies to the truncated Wick ordered SNLW \eqref{SNLW5},
uniformly in $N \in \N$.
More precisely, given 
 $(\phi_0, \phi_1) \in \H^s(\T^2)$
and $N \in \N$, 
there exist
a stopping time $T = T_\o(\phi_0, \phi_1)$ ($>0$ almost surely)
and 
a unique solution $u_N $ to \eqref{SNLW6} on $[0, T]$
such that 
\begin{align*}
u_N \in 
 \Psi_N + C([0, T]; H^\s(\T^2) ) .
\end{align*}

\noi
Moreover, 
one can prove that 
the solutions $u_N$ to \eqref{SNLW6} converges to the solution $u$ to \eqref{SNLW8}
constructed in Theorem \ref{THM:1} as $N \to \infty$.

In the discussion above, we used the Dirichlet projection $\P_N$
onto the spatial frequencies $\{|n|\leq N\}$
for regularization.
The interpretation of the Wick ordered nonlinearity \eqref{eq:wick}
seems to depend on this regularization procedure at this point.

One may instead use a different regularization procedure. 
Given a compactly supported smooth function $\rho \in L^1(\T^2)$ with $\int \rho dx = 1$,
let $\P^\rho_N$ be the mollification given by 
$\P^\rho_N f = \rho_N * f$,
where $\rho_N(x) = N^2\rho(N x)$.
Then, one can consider the regularized stochastic convolution 
$\Psi_{\rho, N}  = \P^\rho_N \Psi$
associated to this mollification
and define the Wick ordered monomials:
\begin{align*}
:\! \Psi^k (x, t) \!: \,      \stackrel{\text{def}}= \lim_{N\to \infty}   H_k(\Psi_{\rho, N}(x, t);\sigma_{\rho,N}(t)), 
\end{align*}

\noi
where $\sigma_{\rho,N}(t) \stackrel{\text{def}}{=} \E[\Psi_{\rho, N}^2(x, t)]$.
By proceeding as in \eqref{sig1}, we have
\begin{align*}
\s_{\rho, N}(t) 
& = \frac{t^3}{3} +  \sum_{ |n| >0} | \ft \rho_N(n)|^2 \bigg\{ \frac{t}{2 |n|^2} 
- \frac{\sin (2 t |n|)}{4|n|^3 }\bigg\} \notag\\
& = \frac{t^3}{3} +  \sum_{ |n| >0} \Big| \F_{\R^2}( \rho)\Big(\frac{n}{N}\Big)\Big|^2 \bigg\{ \frac{t}{2 |n|^2} 
- \frac{\sin (2 t |n|)}{4|n|^3 }\bigg\}, 
\end{align*}

\noi
where  $\F_{\R^2}( \rho)$ is the Fourier transform of $\rho$
when viewed as a function on $\R^2$.
By slightly modifying the proof of Proposition~\ref{PROP:sconv}, 
one can  prove that the Wick ordered monomials
$:\! \Psi^k  \!: $ do not depend on 
the choice of mollifiers (including the convolution kernel of the Dirichlet projection $\P_N$). 
This directly implies that the renormalized nonlinearity $F_\Psi$ is also 
independent of the choice of a mollifier. Of course, the precise value of the renormalization constant 
will depend on $\rho$.

\end{remark}

\begin{remark}
\rm
	
With a small modification of the proof,  
 Theorem \ref{THM:1} also holds for 
the following stochastic nonlinear Klein-Gordon equation 
with an additive space-time white noise:	
\begin{align}
\dt^2 u  +  (1 -  \Dl)  u   \pm  u^k = \xi.
\label{SNLKG}
\end{align}

On the one hand,  we  restrict our attention  to the real-valued setting
in this paper.
On the other hand, it is often useful to consider complex-valued solutions
to the nonlinear Klein-Gordon equation.
We point out that
Theorem \ref{THM:1} also holds in the complex-valued setting, 
provided that we adjust 
the white noise forcing and 
the renormalization procedure
to the complex-valued setting.
In particular, 
one needs to use (generalized) Laguerre polynomials instead of Hermite polynomials.
See Oh-Thomann \cite{OT1} for details.
\end{remark}

\begin{remark} \label{REM:NLW}
\rm 

In the following, we state local well-posedness 
of the following 
deterministic NLW on $\T^2$:
\begin{align}
 \dt^2  u- \Delta u \pm |u|^{k-1} u =0,
 \label{D1}
\end{align}

\noi
where we allow $k\geq2$ to take non-integer values.
We extend the critical regularity $s_\text{crit}$ in \eqref{scrit} 
to a real number $k \geq 2$ by setting
\begin{align}\label{scrit2}
s_\text{crit} := \max\big(s_\text{scaling}, s_\text{conf}, \tfrac 34 - \tfrac3{2k}\big) 
= \max\big(1 - \tfrac{2}{k-1}, 
\tfrac 34 - \tfrac1{k-1},
\tfrac34-\tfrac3{2k}  \big).
\end{align}

\noi
This extends  $s_\text{crit}$ defined \eqref{scrit}
to non-integer values of $k \geq 2$.
As far as we know, the third regularity
$ \frac 34 - \frac3{2k}$ does not correspond to any symmetry of the equation and thus it is not a critical regularity in the usual sense.
It, however, imposes a regularity restriction when $2 \leq k \leq 3$.

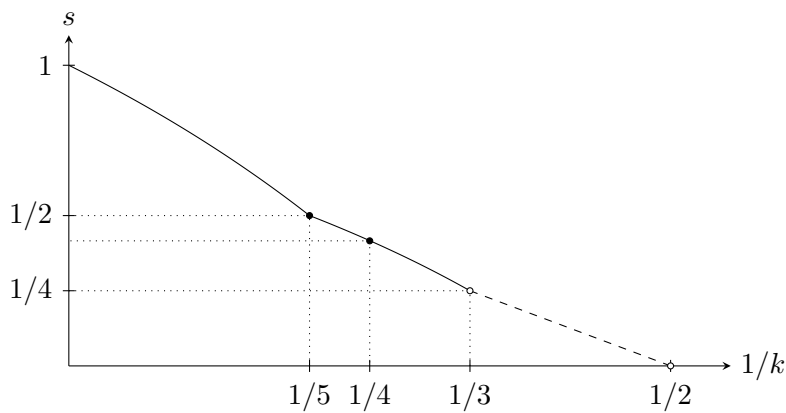
\begin{figure}

\begin{center} 
\begin{tikzpicture}[scale=.8] 
\draw [->] (0,0) -- (11,0) node[anchor=west] {$1/ k$} ; 
\draw [->] (0,0) -- (0,5.5) node[anchor=south]  {$s$} ; 
\draw (.1,1.25)--(-.1,1.25) node[anchor=east] {$1/4$} ; 
\draw (.1,2.5)--(-.1,2.5) node[anchor=east] {$1/2$} ; 
\draw (.1,5)--(-.1,5) node[anchor=east] {$1$}; 
\draw (10 ,.1)--(10 ,-.1) node[anchor=north]{$1/2$};
\draw (6.666 ,.1)--(6.666 ,-.1) node[anchor=north]{$ 1/3$};
\draw (4 ,.1)--(4 ,-.1) node[anchor=north]{$1/5$};

\draw[domain=0:1/5,smooth,variable=\t] plot ({20*\t},{5*(1-(2*\t)/(1-\t))});
\draw[domain=1/5:1/3,smooth,variable=\t] plot ({20*\t},{5*(.75-\t/(1-\t))});

\draw[dashed]  (10,0)--(6.666,1.25);

\draw[dotted]  (6.666,0)--(6.666,1.25); 
\draw[dotted]  (0,1.25)--(6.666,1.25); 

\draw[dotted]  (4,0)--(4,2.5); 
\draw[dotted]  (0, 2.5)--(4,2.5); 

\draw (5 ,.1)--(5 ,-.1) node[anchor=north]{$1/4$};
\draw [dotted] (5, 2.08)--(5, 0); 
\draw [dotted] (5, 2.08)--(0, 2.08); 

\draw (4, 2.5)node[ddot2]{};

\draw (5, 2.08)node[ddot2]{};

\draw (6.666,1.25)node[ddot] {};
\draw (10, 0)node[ddot]{};

\end{tikzpicture} 
\end{center}

\caption{The critical regularity $s_{\text{crit}}$ in \eqref{scrit2} as a function of $\frac 1k$.
The deterministic NLW \eqref{D1} is locally well-posed on and above the solid line
and above the dashed line.}
 \label{FIG:1}

\end{figure}

By the standard Strichartz estimates (see Lemma \ref{LEM:Str})
and a fixed point argument, 
one can easily prove local well-posedness of \eqref{D1}
in $\H^{s} (\T^2)$
for  (i) $s \geq     s_\text{crit}$ if $k >3$ 
and (ii)  $s> s_\text{crit}$ if $ 2 \leq k \leq 3$.
See Subsection \ref{SUBSEC:3.1}.
Figure \ref{FIG:1} shows the range of local well-posedness
of \eqref{D1}
as a function of $\frac{1}{k}$.

\end{remark}

\section{On the stochastic convolution}
\label{SEC:2}

In this section, we establish relevant estimates on 
the stochastic convolution $\Psi$.
In particular, we prove 
the following regularity result  on the Wick ordered monomials 
$:\!\Psi^\l_N(x, t) \!: \, = H_\ell(\Psi_N(x, t),\sigma_N(t))$
defined in \eqref{Herm1}.

\begin{proposition}\label{PROP:sconv}
Let  $\l \in \N$, $T >0$ and $p\ge 1$. 
Then, $\{ :\!\Psi^\l_N \!:  \}_{N \in \mathbb{N}}$ 
is a Cauchy sequence in $L^p(\Omega; C([0,T];W^{-\eps,\infty}(\T^2)))$.
In particular, denoting the limit by  $:\!\Psi^\l\!:$,    
we have  $:\!\Psi^\l\!:\, \in C([0,T];W^{-\eps,\infty}(\T^2))$ almost surely.
\end{proposition}

Before proceeding to the proof of Proposition \ref{PROP:sconv}, 
we recall some basic tools from probability theory and Euclidean quantum field theory.
See \cite{Kuo, Nu, Simon}.
First, 
recall the Hermite polynomials $H_k(x; \s)$
defined through the generating function:
\begin{equation}
F(t, x; \s) \stackrel{\text{def}}{=}  e^{tx - \frac{1}{2}\s t^2} = \sum_{k = 0}^\infty \frac{t^k}{k!} H_k(x;\s).
\label{H1}
 \end{equation}
	
\noi
For simplicity, we set $F(t, x) : = F(t, x; 1)$
and $H_k(x) : = H_k(x; 1)$ in the following.
For readers' convenience, we write out the first few Hermite polynomials:
\begin{align}
\begin{split}
& H_0(x; \s) = 1, 
\qquad 
H_1(x; \s) = x, 
\qquad
H_2(x; \s) = x^2 - \s,   \\
& H_3(x; \s) = x^3 - 3\s x, 
\qquad 
H_4(x; \s) = x^4 - 6\s x^2 +3\s^2.
\end{split}
\label{H1a}
\end{align}
	
\noi
Then, the monomial $x^k$
can be expressed in term of the Hermite polynomials:
\begin{align}
x^k = \sum_{m = 0}^{[\frac{k}{2}]}
\begin{pmatrix}
k\\2m 
\end{pmatrix}
\frac{(2m)!}{2^m m!}\, \s^m H_{k-2m}(x; \s).
\label{H2}
\end{align}

Fix $d \in \N$.\footnote{Indeed, the discussion presented here also holds for $d = \infty$ 
in the context of abstract Wiener spaces.
For simplicity, however, we consider only finite values for $d$.} Consider the Hilbert space $H = L^2(\R^d, \mu_d)$
endowed with  the Gaussian measure 
 $d\mu_d
= (2\pi)^{-\frac{d}{2}} \exp(-{|x|^2}/{2})dx$, $x = (x_1, \dots,
x_d)\in \R^d$. 
Hermite polynomials satisfy
\begin{equation}
\int_{\R} H_k(x)H_m(x) d\mu_1(x)= \delta_{k m} k!
\label{H3}
\end{equation}

\noi
for all  $k,m\in\N$.
Next, we define a {\it homogeneous Wiener chaos of order
$k$} to be an element of the form $\prod_{j = 1}^d H_{k_j}(x_j)$, 
where $k= k_1 + \cdots + k_d$
and $H_{k_j}$ is the Hermite polynomial of degree $k_j$ defined in \eqref{H1}. 
Denote the closure of homogeneous Wiener chaoses of order $k$
under $L^2(\R^d, \mu_d)$ by $\mathcal{H}_k$.
Let $L := \Dl -x \cdot \nabla$ be 
 the Ornstein-Uhlenbeck operator. 
Then, 
it is known that 
any element in $\mathcal H_k$ 
is an eigenfunction of $L$ with eigenvalue $-k$
and that we have the Ito-Wiener decomposition:
\[ L^2(\R^d, \mu_d) = \bigoplus_{k = 0}^\infty \mathcal{H}_k.\]

\noi
Moreover, we have the following  hypercontractivity of the Ornstein-Uhlenbeck
semigroup $U(t) := e^{tL}$ due to Nelson \cite{Nelson2}.

\begin{lemma} \label{LEM:hyp1}
Let $q > 1$
and $p \geq q$.
Then, 
for every $u \in L^q (\R^d, \mu_d)$
and  $t \geq
\frac{1}{2}\log\big(\frac{p-1}{q-1}\big)$, we have
\begin{equation}\label{hyp1}
\|U(t) u \|_{L^p(\R^d, \mu_d)}\leq \|u\|_{L^q(\R^d, \mu_d)}.
\end{equation}
\end{lemma}
We stress that 
 \eqref{hyp1} holds, independent of the dimension $d$.
As a consequence, we obtain the following corollary to Lemma \ref{LEM:hyp1}.
\begin{lemma} \label{LEM:hyp2}
Let $F \in \mathcal{H}_k$.
Then, for $p \geq 2$, we have
\begin{equation} \label{hyp2}
\| F \|_{L^p(\R^d, \mu_d)}\leq (p-1)^\frac{k}{2}
\|F\|_{L^2(\R^d, \mu_d)}.
\end{equation}
\end{lemma}
The estimate \eqref{hyp2} follows immediately
from noting that 
 $F$ is an eigenfunction
of $U(t) = e^{tL}$ with eigenvalue $e^{-kt}$
and setting  $q = 2$ and $t =\frac{1}{2} \log (p - 1)$ in  \eqref{hyp1}.
As a further consequence to Lemma \ref{LEM:hyp2}, we obtain the following lemma
\cite[Theorem I.22]{Simon}.
\begin{lemma}\label{LEM:hyp3}
Fix $k \in \mathbb{N}$ and $c(n_1, \dots, n_k) \in \C$.
Given 	
 $d \in \mathbb{N}$, 
 let $\{ g_n\}_{n = 1}^d$ be 
 a sequence of  independent standard complex-valued Gaussian random variables
 and set $g_{-n} = \cj{g_n}$.
Define $S_k(\o)$ by 
\begin{align*}
 S_k(\o) = \sum_{\G(k, d)} c(n_1, \dots, n_k) g_{n_1} (\o)\cdots g_{n_k}(\o),
 \end{align*}

\noi
where $\G(k, d)$ is defined by
\[ \G(k, d) = \big\{ (n_1, \dots, n_k) \in \{ \pm1, \dots,\pm d\}^k \big\}.\]
Then, for $p \geq 2$, we have
\begin{equation}
 \|S_k \|_{L^p(\O)} \leq (p-1)^\frac{k}{2}\|S_k\|_{L^2(\O)}.
\label{hyp4}
 \end{equation}
\end{lemma}

This  follows from \eqref{H2}
and Lemma \ref{LEM:hyp2}.
Once again, note that \eqref{hyp4} is independent of $d \in \N$.
Lemmas \ref{LEM:hyp2} and \ref{LEM:hyp3}
have been very effective
in the recent probabilistic study of dispersive PDEs
and related areas, see e.g.~\cite{Tzv, TTz, Benyi, CO,  BTT1}.

Lastly, we recall the following property of Wick products \cite[Theorem I.3]{Simon},
extending \eqref{H3} to a more general setting.
See also \cite[Lemma 1.1.1]{Nu}.

\begin{lemma}\label{LEM:Wick}
Let $f$ and $g$ be Gaussian random variables with variances $\s_f$
and $\s_g$.
Then, we have 
\begin{align*}
\E\big[ H_k(f; \s_f) H_m(g; \s_g)\big] = \dl_{km} k! \big\{\E[ f g] \big\}^k.
\end{align*}
\end{lemma}

\begin{proof}[Proof of Proposition~\ref{PROP:sconv}]
First note that it suffices to prove the proposition for large $p \geq 1$,
since $L^{p_1}(\O) \subset L^{p_2}(\O)$ for $p_1 \geq p_2$.
From \eqref{G0},  we have 
\begin{align}
 \mathbb{E} [\Psi_N (t_1, x)  \Psi_N(t_2, y)] 
 = \sum_{n \in \mathbb{Z}^2_N} e_n (x-y)
  \int_0^t \frac{\sin ((t_1 - \tau ) | n |)}{| n |} \frac{\sin ((t_2 - \tau)
   | n |)}{| n |} d\tau ,
\label{G1}
\end{align}

\noi
where $t = \min(t_1, t_2)$.
Define  $\g(n,t)$  by
\[
   \g(n,t) \stackrel{\text{def}}{=}\int_0^t \bigg[ \frac{\sin ((t - \tau) | n |)}{| n |} \bigg]^2 d\tau.
 \]

\noi
By applying the Bessel potentials $\jb{\nb_x}^{-\eps}$ 
and $\jb{\nb_y}^{-\eps}$  of order $\eps$
and then setting $x = y$ (and $t_1 = t_2$), 
we obtain
\[ \E \big[| \jb{\nb}^{-\eps} \Psi_N (x, t) |^2\big] 
= \sum_{n \in \Z^2_N} \jb{n}^{-2\eps} \g(n,t)
   \leq t^3  + t \sum_{n \in \Z^2_N}
   \frac{1}{\jb{n}^{2+2\eps}} \les t^3 + t \]

\noi
for any $\eps > 0$, $x \in \T^2$, and $t > 0$, 
uniformly in $N \in \mathbb{N}$.
In particular, by the hypercontractivity (Lemma \ref{LEM:hyp3}), we have
\[ \E\big[ | \jb{\nb}^{-\eps} \Psi_N (t, x) |^p\big] \les_{p, t} 1 \]

\noi
and thus
\[\E \big[\| \Psi_N (\cdot, t) \|_{W^{-\eps,p}}^p \big]
=  \E \big[\| \jb{\nb}^{-\eps} \Psi_N (\cdot, t) \|_{L^p (\T^2)}^p \big]
< \infty \]

\noi
for any $\eps > 0$, $t > 0$,  and $p \geq 1$, uniformly in $N\in \mathbb{N}$. 

By Lemma \ref{LEM:Wick} 
and \eqref{G1},  we have
\begin{align*}
 \E \big[:\! \Psi^{\l}_N (x, t) \!:\,  :\! \Psi^{\l}_N (y, t) \!:\big] 
 & = \l! \big\{\mathbb{E} [\Psi_N (x, t) \Psi_N (y, t)]\big\}^{\l} \\
  & =\l! \sum_{n_1, \ldots, n_{\ell} \in \mathbb{Z}^2_N} \gamma (n_1, t) \cdots
  \gamma (n_{\ell}, t) e_{n_1} (x - y) \cdots e_{n_{\ell}} (x - y) \\
  & = \l!\sum_{n_1, \ldots, n_{\ell} \in \mathbb{Z}^2_N} \gamma (n_1, t) \cdots
  \gamma (n_{\ell}, t) e_{n_1 + \cdots + n_{\ell}} (x - y) .
\end{align*}

\noi
Proceeding as before, we obtain
\begin{align*}
\E\big[| (\jb{\nb}^{-\eps} :\! \Psi^{\l}_N (\cdot, t) \!:) (x) |^2 \big] 
& = \l! \sum_{n_1,
   \ldots, n_{\ell} \in \mathbb{Z}^2_N} \langle n_1 + \cdots + n_{\ell}\rangle^{-2\eps} 
   \gamma (n_1, t) \cdots \gamma (n_{\ell}, t) \\
& \les_t \sum_{n_1, \ldots, n_{\ell} \in \mathbb{Z}^2} \frac{1}{\langle n_1
   \rangle^2 \cdots \langle n_{\ell} \rangle^2 \langle n_1 + \cdots + n_{\ell}
   \rangle^{2\eps}} < \infty
\end{align*}

\noi
\noi
for any $\eps > 0$, $x \in \T^2$, and $t > 0$, 
uniformly in $N$.
Hence,  by the hypercontractivity (Lemma \ref{LEM:hyp3}), 
we have
\begin{align*}
\E\big[\|  :\! \Psi^{\l}_N (\cdot, t) \!:\|_{W^{-\eps, p}}^p \big] 
< \infty
\end{align*}

\noi
for any $\eps > 0$, $t > 0$,  and $p \geq 1$, uniformly in $N\in \mathbb{N}$.

In order to analyze the time regularity, 
we have to estimate moments of the random field
\[\delta_h \! :\!\Psi^{\ell}_N (x, t) \!: \, \, \stackrel{\text{def}}{=}\,  \, : \!\Psi^{\ell}_N (x, t + h) \!: - :\! \Psi_N^{\ell} (x, t)\! :\] 

\noi
for $h \in [-1, 1]$.
In the following, we proceed as above and estimate
\[\mathbb{E} \big[|\dl_h (\jb{\nb}^{-\eps} : \! \Psi_N^{\ell} (\cdot, t ) \!:) (x) |^2 \big].\]

\noi 
By applying Lemma \ref{LEM:Wick} once again, we have
\begin{align*}
\frac{1}{\l!}\E \big[\delta_h \! :\! \Psi^{\ell}_N & (x, t) \!: \,  \delta_h\! :\! \Psi^{\ell}_N (y, t)\!:  \big]\notag\\
 & =  \big\{\mathbb{E} [\Psi_N (x, t + h) \Psi_N (y, t + h)]\big\}^\l 
    - \big\{\E [\Psi_N ( x, t) \Psi_N (y, t + h)]\big\}^\l \\
  & \hphantom{XX}
    - \big\{\E [\Psi_N (x, t + h) \Psi_N (y, t)]\big\}^\l 
+   \big\{\E [\Psi_N (x, t) \Psi_N (y, t)]\big\}^\l \nonumber\\
  & = \mathbb{E} [\dl_h \Psi_N ( x, t) \Psi_N (y, t + h)] \\
& \hphantom{XX}
\times   \sum_{j = 0}^{\l-1}    \big\{\mathbb{E} [ \Psi_N ( x, t+h) \Psi_N (y, t + h)]\big\}^{\l-j-1} 
   \big\{\mathbb{E} [ \Psi_N ( x, t) \Psi_N (y, t + h)]\big\}^j \\
& \hphantom{X}
- \mathbb{E} [\dl_h \Psi_N ( x, t) \Psi_N (y, t )] \\
& 
\hphantom{XX}
\times   \sum_{j = 0}^{\l-1}    \big\{\mathbb{E} [ \Psi_N ( x, t+h) \Psi_N (y, t )]\big\}^{\l-j-1} 
   \big\{\mathbb{E} [ \Psi_N ( x, t) \Psi_N (y, t )]\big\}^j .
\end{align*}

By reasoning as before, in order to estimate 
$\mathbb{E} \big[|\dl_h (\jb{\nb}^{-\eps} : \! \Psi^{\ell} (\cdot, t ) \!:) (x) |^2 \big]$,
we are led to bound
sums of the form
\begin{align}
 S_{h, \eps} = \sum_{n_1, \ldots, n_{\ell} \in \mathbb{Z}^2_N} \langle n_1 + \cdots +
   n_{\ell}\rangle^{-2\eps} G_1 (n_1, t, h) \cdots G_{\ell} (n_{\ell}, t, h),  
\label{G2}
\end{align}

\noi
where $G_i(n, t)$ is given by 
\begin{align*}
G_1 (n, t) & =\E [\dl_h \ft \Psi_N (n, t)  \ft{\Psi}_N ( n, t_1)], \\
G_i (n, t) & =\E [\ft \Psi_N (n, t_1) \ft \Psi_N (n, t_2)], \quad i = 2, \dots, \l, 
\end{align*}

\noi
with $t_1, t_2 \in \{t, t+h\}$.
Here, $\ft \Psi_N (n, t)$ denotes the spatial Fourier transform of $\Psi_N(t)$.  
A direct computation with \eqref{G0} gives 
\begin{align} 
\big| \E [\ft \Psi_N (n, t_1) \ft \Psi_N (n, t_2)] \big| 
\les_t \frac{1}{\jb{n}^2}
\qquad \text{and}\qquad 
\big| \E [\dl_h \ft \Psi_N (n, t) \ft {\Psi}_N (n, t_1)] \big| 
\les_t \frac{|h|^{\rho}}{\jb{n}^{2 -\rho}} 
\label{G3}
\end{align}

\noi   
for any $\rho \in [0, 1]$,
where the implicit constants are independent of  $h \in [-1, 1]$.
Note that the second estimate follows
from interpolating 
\begin{align*} 
\big| \E [\dl_h \ft \Psi_N (n, t) \ft {\Psi}_N (n, t_1)] \big| 
\les_t \frac{1}{\jb{n}^2}
\qquad \text{and}
\qquad 
\big| \E [\dl_h \ft \Psi_N (n, t) \ft {\Psi}_N (n, t_1)] \big| 
\les_t \frac{|h|}{\jb{n}}, 
\end{align*}

\noi
where the second bound follows from the mean value theorem.
As a consequence, it follows from \eqref{G2} and \eqref{G3} that 
\[ | S_{h, \eps} | \les  |h|^\rho \]

\noi
for any $h \in [-1,1]$, $\eps > 0$,  and  $\rho \in [0, 1]$ such that $2 \eps -\rho >0$. 
This in turn implies that
\[ \E \big[| \dl_h ( \jb{\nb}^{-\eps} : \!\Psi^{\l}_N (\cdot, t) \!:) (x) |^2 \big]
   \les  |h|^{\rho}. \]

\noi
Then, by the hypercontractivity (Lemma \ref{LEM:hyp3}), 
this results in 
\begin{align*}
 \E \Big[\big\| \delta_h (: \!\Psi^{\ell}_N (\cdot, t) \!:) 
\big\|_{W^{- \eps,p}}^p \Big]
\les_{p, t} |h|^{\frac{p}{2}\rho}, 
\end{align*}

\noi
for any $h \in [-1,1]$, $\rho \in [0,1]$, and $\eps>0$ such that $2\eps > \rho$.
Hence,  it follows from  Sobolev's embedding theorem that,
given $\eps > 0$, we have
\begin{align*}
 \E \Big[\big\| \delta_h (: \!\Psi^{\ell}_N (\cdot, t) \!:) 
\big\|_{W^{- \eps,\infty}}^p \Big]
\les 
 \E \Big[\big\| \delta_h (: \!\Psi^{\ell}_N (\cdot, t) \!:) 
\big\|_{W^{- \frac{\eps}{2},p}}^p \Big]
\les_{p, t} |h|^{\frac{p}{2}\rho}, 
\end{align*}

\noi
for $p$ sufficiently large such  that $\eps p > 4$.
Moreover, for fixed  $\rho \in (0, 2\eps)$, 
we can choose $p \gg1 $ such that $\frac p2 \rho > 1$,
allowing us to apply  Kolmogorov's continuity criterion (see \cite[Theorem 8.2]{Bass})
and conclude that 
$ :\!\Psi^\l_N \!:\,  \in C([0, T]; W^{-\eps, \infty}(\T^2))$
almost surely,  for any $T > 0$ and $\eps > 0$.

A similar argument also leads  to the following estimate:
\[ \E \big[| \dl_h (\jb{\nb}^{-\eps} (:\! \Psi^{\ell}_N (\cdot, t)\! :  -  :\! \Psi^{\ell}_M (\cdot, t) \!:)) (x) |^2 \big]
   \les_t  \frac{|h|^{\rho}}{N^{2\kappa}} \]

\noi
for all  $M\ge N \ge 1$, $\kappa >0$,  $\eps >0$,
and $\rho\in[0,1]$
 such that $2\eps - 2\kappa - \rho>0$.
By the hypercontractivity (Lemma \ref{LEM:hyp3}), 
this results in 
\[ \E \Big[\big\| \delta_h (: \!\Psi^{\ell}_N (\cdot, t) \!: - : \!\Psi^{\l}_M (\cdot, t) \!:) 
\big\|_{W^{- \eps,p}}^p \Big]
\les_{p, t} \frac{|h|^{\frac{p}{2}\rho}}{N^{\kappa p}}, \]

\noi
for any $\rho \in [0,1]$ and $\eps,\kappa >0$ such that 
\begin{align}
2\eps > \rho + 2\kappa.
\label{G4}
\end{align} 

\noi
As before, 
by Sobolev's embedding theorem
and 
Kolmogorov's continuity criterion, 
 we deduce that for any $T > 0$ and $\eps > 0$,
  there exists large $p\gg1$ such that 
$\{ :\!\Psi^\l_N \!:  \}_{N \in \mathbb{N}}$ 
is a Cauchy sequence in $L^p(\Omega;  C ([0, T] ; W^{- \eps,\infty}(\T^2)))$. 
Denoting  the corresponding limit
by $: \!\Psi^{\ell} \!:$ as in \eqref{Wick2}, 
we conclude that  $: \!\Psi^{\ell} \!:\, \in C ([0, T] ; W^{- \eps,\infty}(\T^2))$ almost surely.
\end{proof}

\begin{remark} \rm
From the application of Kolmogorov's continuity criterion (see \cite[Exercise 8.2]{Bass}), 
we see that 
$:\! \Psi^{\ell}\!\! : \, \in C^\al  ([0, T] ; W^{- \eps,\infty}(\T^2))$,
$\al < \frac{\rho}{2} - \frac 1p$,  almost surely,
provided that \eqref{G4} is satisfied.
In particular,  by taking $p \to \infty$ and $\kappa \to 0$, 
we see that $\al + (- \eps) < 0$,
namely, the sum of the temporal and spatial regularities must be negative.

\end{remark}

\section{Proof of Theorem \ref{THM:1}}
\label{SEC:3}

In this section, we present the proof of Theorem \ref{THM:1}.
In particular, we study the fixed point problem
\eqref{SNLW9}
by constructing a pathwise contraction
in a suitable function space.

\subsection{Strichartz estimates}
\label{SUBSEC:3.1}

We first recall the Strichartz estimates
for the  linear wave equation.
Given  $0 < s < 1$, 
we say that a pair $(q, r)$ is $s$-admissible 
(a pair $(\wt q, \wt r)$ is dual $s$-admissible,\footnote{Here, we define
the notion of dual $s$-admissibility for the convenience of the presentation.
Note that $(\wt q, \wt r)$ is dual $s$-admissible
if and only if $(\wt q', \wt r')$ is $(1-s)$-admissible.}
 respectively)
if $1 \leq \wt q < 2 < q \leq \infty$, 
 $1< \wt r \le 2 \leq r < \infty$, 
\begin{align}
 \frac{1}{q} + \frac 2r  = 1-  s =  \frac1{\wt q}+ \frac2{\wt r} -2, 
\qquad
\frac 2q + \frac{1}{r} \leq \frac 1 2, 
\qquad \text{and} 
\qquad  
\frac2{\wt q}+\frac1{\wt r} \geq \frac52   .
\label{admis1}
\end{align}

\noi
We refer to the first two equalities as the scaling conditions
and the last two inequalities as the admissibility conditions.

Let us now state a lemma, providing a more direct description of the admissible exponents.

\begin{lemma} \label{LEM:pair} 
Let $0< s<1$.
A pair $(q,r)$ is $s$-admissible if 
\begin{align}
  \frac{1}{q} + \frac 2r  = 1-  s 
\qquad  
\text{and} \qquad
2\le r \le 
\begin{cases}
 \frac6{3-4s}, & \text{if } s < \frac34, \\ 
\infty, & \text{otherwise}. 
\end{cases}
\label{admis2}
\end{align}

\noi
A pair $(\wt  q,\wt  r)$ is dual $s$-admissible if 
\begin{align} 
\frac1{\wt q} + \frac2{\wt r} = 3 - s 
\qquad  
\text{and} \qquad
\max\bigg\{1+, \frac6{7-4s} \bigg\}  \le   \wt r \le \frac2{2-s}. 
\label{admis3}
\end{align}

\end{lemma}

We say that  $u$ is a solution to the following nonhomogeneous linear wave equation:
\begin{align}
\begin{cases}
\dt^2 u -\Delta u = f \\
( u, \dt u) |_{t = 0}=(\phi_0,  \phi_1) 
\end{cases}
\label{NLW1}
\end{align}

\noi
on a time interval containing $t= 0$, 
if $u$ satisfies the following Duhamel formulation:
\[ u =  \cos (t|\nb|) \phi_0 +   \frac{\sin (t|\nb|)}{ |\nb|} \phi_1 
+ \int_0^t  \frac{\sin ((t-\tau)|\nb|)}{ |\nb|}   f(\tau) d\tau . \]  

\noi
We now recall the  Strichartz estimates
for solutions to the nonhomogeneous linear wave equation
\eqref{NLW1}.

\begin{lemma}\label{LEM:Str}
Given $0 < s < 1$,
let $(q, r)$ and $(\wt q,\wt r)$
be $s$-admissible and dual $s$-admissible pairs, respectively. 
Then, a solution $u$ to the nonhomogeneous linear wave equation \eqref{NLW1}
satisfies
\begin{align}
\| (u, \dt u) \|_ {L^\infty_T \H^s } + 
 \| u  \|_{L^q_TL^r_x}
\lesssim 
\|(\phi_0, \phi) \|_{\H^s}  +  \| f \|_{L^{\wt q}_TL^{\wt r}_x},
\label{Str1}
\end{align}

\noi
for all $0 < T \leq 1$. 
The following estimate also holds:
\begin{align}
\| (u, \dt u) \|_ {L^\infty_T \H^s } 
+  \| u  \|_{L^q_TL^r_x}
\lesssim 
\|(\phi_0, \phi) \|_{\H^s}  +  \| f \|_{L^{1}_T H^{s-1}_x},
\label{Str2}
\end{align}

\noi
for all $0 < T \leq 1$.
Here, we used a shorthand notation
$L^q_TL^r_x$ = $L^q([0, T]; L^r(\T^2))$, etc.
\end{lemma}

The Strichartz estimates on $\R^d$ have been studied by many
mathematicians.  See Ginibre-Velo \cite{GV}, Lindblad-Sogge \cite{LS},
and Keel-Tao \cite{KeelTao}. 
 The first estimate \eqref{Str1} on $\T^2$ in Lemma \ref{LEM:Str} follows
from  Theorem 2.6 in 
\cite{TAO} for $\R^2$ 
and   the finite speed of propagation for the wave equation.
The first term on the left-hand side of  the second estimate \eqref{Str2} 
is estimated by   the energy estimate (2.29) in \cite{TAO}
and   the finite speed of propagation for the wave equation, 
while 
the second term on the left-hand side of the second estimate \eqref{Str2} 
is estimated by   Minkowski's integral inequality
and the homogeneous Strichartz estimate  in \eqref{Str1}: 
\begin{align*}
 \bigg\| \int_0^t \frac{\sin((t-\tau)|\nb|)}{|\nb|}f(\tau)d\tau  \bigg\|_{L^q_TL^r_x}  
&  \le   \int_0^T  \bigg\| \ind_{[0, t]}(\tau) \frac{\sin((t-\tau)|\nb|)}{|\nb|} f(\tau) \bigg\|_{L^q_t([0, T]; L^r_x)} d\tau
\\ 
& \les   \int_0^T \| f(\tau) \|_{H^{s-1}} d\tau. 
\end{align*}

In the remaining part of this subsection, we consider  
 the following  deterministic  wave equation with $k \geq 2$: 
\begin{align}
 \dt^2  u- \Delta u \pm |u|^{k-1} u =0. 
 \label{HK0}
\end{align}

\noi
Here, we allow $k \geq 2$ to take non-integer values.
In particular, we prove local well-posedness of \eqref{HK0}
in $\H^{s} (\T^2) = H^s(\T^2) \times H^{s-1}(\T^2)$ 
with  (i) $s \geq     s_\text{crit}$ if $k >3$ 
and (ii) $s> s_\text{crit}$ if $ 2 \leq k \leq 3$, 
where 
 $ s_\text{crit}$ is the  regularity defined in \eqref{scrit2}.

Suppose that we can find
an $s$-admissible pair $(q,r)$ and a dual $s$-admissible pair $(\wt q,\wt r)$ so that 
\begin{align}
q \ge k \wt q \qquad \text{and} \qquad  r\ge k \wt r. 
\label{HK1}
\end{align}

\noi
Then,  H\"older's inequality with the fact that $|\T^2| = 1$ yields
\begin{align}
  \big\| |u|^{k-1} u \big\|_{L^{\wt q}_T L^{\wt r}_x}
  \le T^{\frac{1}{\wt q}- \frac{k}q} \|  u \|_{L^q_T L^r_x}^k. 
\label{HK2}
\end{align}

\noi
Then, local well-posedness of \eqref{HK0} on 
a time interval $[0, T]$ for some $T = T(\phi_0, \phi_1)> 0$
follows
from  the Strichartz estimates (Lemma \ref{LEM:Str}), 
 \eqref{HK2}, and a standard contraction argument.
Note that we have a positive power of $T$ in \eqref{HK2}
when $q > k \wt q$. 
In this case, we can take $T = T(\| (\phi_0, \phi_1) \|_{\H^s}) > 0$.
Indeed, this is the case when $s$ is greater than the scaling critical regularity $s_\text{scaling}$.

\medskip

Fix  $0<s<1$.
Then, in view of \eqref{HK1},  we would like to  maximize 
\[   \min\Big\{ \frac{q}{\wt q}, \frac{r}{\wt r}  \Big\} \]

\noi
under the constraints of  Lemma \ref{LEM:pair}. 
While this is non-inspiring and can be easily done, 
the result gives important insights. 
In view of \eqref{admis1}, 
this is essentially\footnote{Here, we allow $\wt r =1$ that is not admissible for the Strichartz estimates.} equivalent 
to  the following maximization problem on 
$ J_s(r,\wt r)$ defined by 
\begin{align}
 J_s(r,\wt r) = 
 \frac{r}{\wt r} \min\left\{ 1, \frac{(3-s)\wt r-2 }{(1-s)r -2 } \right\} 
\label{J1}
\end{align}

\noi
over the set 
\begin{align}
 K(s)= \bigg[ 2,\frac{6}{(3-4s)_+}\bigg] \times \bigg[\max\Big\{1,\frac{6}{7-4s}\Big\}, \frac{2}{2-s}\bigg],   
\label{J2}
\end{align}

\noi
where $x_+ : = \max(x, 0)$ 
with the understanding that $\frac60 = \infty$.

\begin{lemma}\label{LEM:max} 
Given $0<s<1$, let 
$ J_s(r,\wt r)$ and 
$K(s)$ be as in \eqref{J1} and \eqref{J2}.
Then, the maximum of  $ J_s(r,\wt r)$  on $K(s)$ is given by 
\[  \sup_{(r, \wt r) \in K(s)}  J_s(r, \wt r)
= \begin{cases}
\rule[-3mm]{0pt}{0pt}
 \frac{3-s}{1-s}, & \text{if }  \frac12 \leq s < 1, \\
\rule[-3mm]{0pt}{0pt}
 \frac{7-4s}{3-4s}, & \text{if } \frac14 \le s \le \frac12, \\
\rule[-2.5mm]{0pt}{0pt}
 \frac{6}{3-4s}, & \text{if } 0 < s \le \frac14. 
\end{cases}
\]

\noi
Moreover, the supremum is indeed attained in each case:
\textup{(i)} when $0 < s \le \frac14$, it is attained at $(r, \wt r) = \big( \frac{6}{3-4s}, 1\big)$, 
\textup{(ii)} when $\frac14\le s \le \frac12$, it is attained at
$(r, \wt r) = \big( \frac{6}{3-4s}, \frac{6}{7-4s}\big)$, 
and
\textup{(iii)}
when $ \frac 12 \leq s < 1$,  it is attained in the set:
\begin{align}
  \frac{6}{7-4s}\cdot \frac{3-s}{1-s}
\le r \le 
\begin{cases}
\rule[-3mm]{0pt}{0pt}
\frac{6}{3-4s}, &  \text{if } \frac12 \le s \le 3-\sqrt{6}\sim 0.55,   \\
\frac{2}{2-s} \cdot \frac{3-s}{1-s}, & \text{if }  3-\sqrt{6}    \leq s < 1, 
\end{cases}
\quad \text{and}
\quad 
\wt r = \frac{1-s}{3-s} r.
\label{J3}
\end{align}

\end{lemma}

 \begin{proof} 
From \eqref{J1}, we see that 
the maximum of $J_s(r, \wt r)$ on $K(s)$ is given by $\max\{J_1(s),J_2(s) \}$, 
where 
\[ J_1(s)= \max \bigg\{ \frac{r}{\wt r} : \frac{r}{\wt r} \le \frac{3-s}{1-s}, \, (r,\wt r) \in K(s) \bigg\} \] 
and 
\[J_2(s) = \max \bigg\{ \frac{3-s-\frac2{\wt r}}{1-s-\frac2{r}} :   
 \frac{r}{\wt r} \ge \frac{3-s}{1-s}, \, (r,\wt r ) \in K(s) \bigg\}. \]

From \eqref{admis2} and \eqref{admis3}, we have
\[ \sup\bigg\{ \frac{r}{\wt r}: (r,\wt r) \in K(s) \bigg\} 
=  \begin{cases}
\rule[-2.5mm]{0pt}{0pt}
 \frac{6}{3-4s},        & \text{ if } 0<s \le \frac14,  \\
\rule[-2mm]{0pt}{0pt}
\frac{7-4s}{3-4s},    & \text{ if } \frac14 \le s \le \frac34,  \\
\infty,  & \text{ if } s \ge \frac34
\end{cases}
\]

\noi
and  
\[ \min\bigg\{ \frac{r}{\wt r}: (r,\wt r) \in K(s)\bigg\} =  2-s . \]

\noi
Note that we have 
$ \frac{6}{3-4s} \le \frac{3-s}{1-s} $ for $ s \le \frac14$ and 
$\frac{7-4s}{3-4s} \le \frac{3-s}{1-s}$ for $ s\le \frac12$.
Hence,  for $0<  s\le \frac12 $, we have
 \[  \sup_{(r, \wt r) \in K(s)}  J_s(r, \wt r)
=   \max\{J_1(s), J_2(s)\} = \sup\bigg\{ \frac{r}{\wt r} : (r,\wt r) \in K(s)\bigg\}.   \]

Next, we consider the case  $\frac12 < s < 1$. 
On the one hand, we have  $J_1(s) \le \frac{3-s}{1-s}$.
On the other hand, by minimizing $r$ and maximizing $\wt r$ under   $\frac{r}{\wt r} \ge \frac{3-s}{1-s}$,
we obtain
\begin{align*}
J_2(s)
& =  \max\left\{ \frac{3-s-\frac2{\wt r}}{1-s-\frac2r}:\frac{r}{\wt r} = \frac{3-s}{1-s} , \,  (r,\wt r) \in K(s) \right\}
\\ 
& =  \frac{3-s}{1-s}. 
\end{align*}

\noi
It is easy to check that  this maximum is attained 
in the set described in \eqref{J3}.
\end{proof}

As a result,  we can prove local well-posedness of the deterministic NLW
\eqref{D1} at the regularities stated in Remark \ref{REM:NLW}.
Indeed, it suffices to note that 
Lemma \ref{LEM:max} guarantees the existence of 
 an $s$-admissible pair $(q, r)$ and a dual $s$-admissible pair $(\wt q, \wt r)$
 satisfying \eqref{HK1},
 provided that 
  (i) $s \geq     s_\text{crit}$ if $k >3$ 
and (ii) 
and $s> s_\text{crit}$ if $ 2 \leq k \leq 3$, 
where $s_\text{crit}$ is as in \eqref{scrit2}.
Note that when $ 2 \leq k \leq 3$, 
 the endpoint $s = s_\text{crit}$ is excluded 
 since the maximum in Lemma \ref{LEM:max} is attained
 at $\wt r = 1$, which is not allowed for  the dual $s$-admissibility.
 Then, the rest of the proof of the local well-posedness follows
 from  the Strichartz estimates (Lemma~\ref{LEM:Str}), \eqref{HK2}, and 
a standard fixed point argument.
See also the discussion in Subsection~\ref{SUBSEC:3.3}.

\subsection{Estimating a product}

In this subsection, we state several product estimates
for periodic functions on $\T^d$.
First, recall the following fractional Leibniz rule
for functions on $\R^d$;
let 
 $1<p_j,q_j,r < \infty$, $\frac1{p_j} + \frac1{q_j}= \frac1r$, $j = 1, 2$. 
 Then, we have  
\begin{equation} 
\big\| |\nb|^s (fg) \big\|_{L^r(\R^d)} 
\les  \big\| f \big\|_{L^{p_1}(\R^d)} \big\| |\nb|^s g \big\|_{L^{q_1}(\R^d)} 
+ \big\| |\nb|^s f \big\|_{L^{p_2}(\R^d)}  \big\| g \big\|_{L^{q_2}(\R^d)}. 
\label{bilinear1}
\end{equation}  

\noi
This estimate is an immediate consequence of 
 the Coifman-Meyer theorem; see \cite{CM} and the inequality (1.1) in \cite{MS}.    
We use \eqref{bilinear1} to prove  the following product estimates
for functions on $\T^d$.

\begin{lemma}\label{LEM:bilin}
 Let $0\le s \le 1$.

\smallskip

\noi
\textup{(i)} Suppose that 
 $1<p_j,q_j,r < \infty$, $\frac1{p_j} + \frac1{q_j}= \frac1r$, $j = 1, 2$. 
 Then, we have  
\begin{equation}  
\| \jb{\nb}^s (fg) \|_{L^r(\T^d)} 
\les \Big( \| f \|_{L^{p_1}(\T^d)} 
\| \jb{\nb}^s g \|_{L^{q_1}(\T^d)} + \| \jb{\nb}^s f \|_{L^{p_2}(\T^d)} 
\|  g \|_{L^{q_2}(\T^d)}\Big).
\label{bilinear5}
\end{equation}

\smallskip

\noi
\textup{(ii)} 
Suppose that  
 $1<p,q,r < \infty$ satisfy the scaling condition:
$\frac1p+\frac1q=\frac1r + \frac{s}d $.
Then, we have
\begin{align}
\big\| \jb{\nb}^{-s} (fg) \big\|_{L^r(\T^d)} \les \big\| \jb{\nb}^{-s} f \big\|_{L^p(\T^d) } 
\big\| \jb{\nb}^s g \big\|_{L^q(\T^d)}.  
\label{bilinear2}
\end{align}

\end{lemma} 

\begin{proof} 
In view of the transference principle \cite[Theorem 3]{FS}, 
the first estimate \eqref{bilinear5} follows from the Coifman-Meyer theorem 
for functions on $\R^d$
and \eqref{bilinear1}.
 The second estimate \eqref{bilinear2} follows 
 from  duality, the first estimate \eqref{bilinear5}, and Sobolev's inequality:
\begin{align*}
  \big\| \jb{\nb}^{-s} (fg) \big\|_{L^r} 
  & \le  \sup_{\|\jb{\nb}^s h \|_{L^{r'}} =1}  \bigg|\int fgh \, dx \bigg|\\
 & \le \big\| \jb{\nb}^{-s} f \big\|_{L^p}   \sup_{\| \jb{\nb}^s h \|_{L^{r'}} =1} 
 \big\| \jb{\nb}^{s} (gh) \big\|_{L^{p'}} \\
 &  \les \big\| \jb{\nb}^{-s} f \big\|_{L^p}  
 \sup_{\| \jb{\nb}^s h \|_{L^{r'}} = 1} \Big( \| g \|_{L^{\wt  q}} \big\| \jb{\nb}^s h \big\|_{L^{r'}} 
+ \big\|\jb{\nb}^s g \big\|_{L^q} \| h \|_{L^{\wt r'} } \Big) \\
 &  \lesssim   \big\| \jb{\nb}^{-s} f \big\|_{L^p} \big\| \jb{\nb}^s g \big\|_{L^q}   , 
\end{align*}

\noi
where the exponents satisfy the H\"older relations: 
\begin{align}
 \frac1q + \frac1{\wt r'} = \frac1{\wt q} + \frac1{r'}=\frac{1}{p'} 
\label{bilinear3}
\end{align}

\noi
and 
the exponents satisfy the Sobolev relations: 
\begin{align}
 \frac1{\wt q} = \frac1q -\frac{s}{d}  
  \qquad \text{and}\qquad
  \frac1{\wt r'} = \frac1{r'} - \frac{s}d.  
\label{bilinear4}
\end{align}

\noi
Altogether, \eqref{bilinear3} and \eqref{bilinear4} yield the scaling condition. 
\end{proof}

\subsection{Local well-posedness of SNLW}
\label{SUBSEC:3.3}

In this subsection, we present the proof of Theorem \ref{THM:1}.
Given an integer $k \geq 2$ and $(\phi_0, \phi_1) \in \H^s(\T^2)$, define a map $\G$ by 
\begin{align}
v \mapsto \G(v)(t)
& \stackrel{\text{def}}{=} 
S(t) (\phi_0, \phi_1)
 \mp \int_{0}^t  \frac{\sin ((t-\tau)|\nabla|)}{|\nabla|}
    F_\Psi(v(\tau)) d\tau \notag \\
& = S(t) (\phi_0, \phi_1)
 \mp  \sum_{\ell=0}^k {k\choose \ell} \int_{0}^t  \frac{\sin ((t-\tau)|\nabla|)}{|\nabla|}
 :\! \Psi^\l (\tau)  \!: v^{k-\l}(\tau ) d\tau.
\label{SNLW10}
\end{align}

\noi 
Let  $s $ be as in Theorem \ref{THM:1}. 
More precisely we assume that  
(i) $ s\geq s_\text{crit} $ if $k \ge 4$, 
(ii)~$ s > \frac14$ if $k=3$, and 
(iii) $s>0$ if $k=2$. 
 In the following, we only consider the case $s < 1$.

In view of Lemma \ref{LEM:max} and \eqref{scrit}, 
we can choose an $s$-admissible pair $(q, r)$ and a dual $s$-admissible pair $(\wt q, \wt r)$
such that 
\begin{align*}
\min\Big\{  \frac{q}{\wt q} ,\frac{r}{\wt r} \Big\} \ge k 
\end{align*}

\noi
with a strict inequality if $k=2$ or $3$.

We define $X^s(T)$ as the intersection of the energy space at level $s$ and the Strichartz space
\begin{align*}
 X^s(T) = C([0,T];H^s(\T^2))\cap C^1([0,T]; H^{s-1}(\T^2))\cap L^q([0, T]; L^r(\T^2)). 
\end{align*}

\begin{proposition}\label{PROP:nonlin} 
Given an integer $k \geq 1$, let $s$, $(q, r)$, and $(\wt q, \wt r)$ be as above.
Then,  there exist sufficiently small $\eps >0$ and $\theta > 0$ such that 
\begin{align} 
\| \G(v) \|_{X^s(T)} 
& \les
\|(\phi_0, \phi_1)\|_{\H^s}
+ \| :\!\Psi^k\!: \|_{L^1_TH^{s-1}_x} \notag\\
& \hphantom{XXX}+   T^\theta \sum_{\l=1}^{k-1} \|\jb{\nb}^{-\eps} :\!\Psi^\l\!: \|_{L^\infty_{T, x}} 
\| v \|_{X^s(T)}^{k-\l}  
 +  T^{\frac{1}{\wt q}-\frac kq} 
\| v \|_{X^s(T)}^k   
\label{J3a}
\end{align}
and 
\begin{align}
\| \G(v_1)-\G(v_2)  \|_{X^s(T)} & \les  
T^\theta \sum_{\l=1}^k \|\jb{\nb}^{-\eps} :\!\Psi^\l\!: \|_{L^\infty_{T, x}} \notag\\
& \hphantom{XXXX}
\times
\Big( \| v_1 \|_{X^s(T)}  + \| v_2 \|_{X^s(T)}\Big)^{k-\l-1}  
\| v_1 - v_2 \|_{X^s(T)}   \notag\\ 
  & \hphantom{XXXX}
  +  T^{\frac{1}{\wt q}-\frac k q} 
\Big( \| v_1 \|_{X^s(T)}+ \| v_2 \|_{X^s(T)}\Big)^{k-1} \| v_1-v_2 \|_{X^s(T)}
\label{J3b}
\end{align}

\noi
for any $T > 0$.

\end{proposition}

\begin{proof} 
We only prove the first estimate \eqref{J3a} since the second estimate
follows in a similar manner.
As in Subsection \ref{SUBSEC:3.1}, 
we can estimate the term with $\l=0$ in \eqref{SNLW10} 
by Lemma \ref{LEM:Str} and \eqref{HK2}.
On the other hand, we can use \eqref{Str2} in Lemma \ref{LEM:Str}
to estimate 
the first term on the right-hand side of \eqref{J3a}
and 
the term with $\l=k$ in \eqref{SNLW10}.
Hence, it remains to prove
\begin{equation} 
\label{multilinear}  
\bigg\| \int_0^t \frac{\sin ((t-\tau )|\nb|)}{|\nb|} :\!\Psi^\l\!: \prod_{j = 1}^{k-\l}v_j  d\tau \bigg\|_{X^s(T)} 
\les  T^\theta \|\jb{\nb}^{-\eps} :\!\Psi^\l\!: \Vert_{L^\infty_{T, x}} \prod_{j=1}^{k-\l}  \| v_j \|_{X^s(T)}
\end{equation} 

\noi
for $1\leq \l\leq k-1$.
To simplify the notation,  we only consider the case  $v_j=v$ in the following. 
The full estimate can be recovered by polarization or, what may be easier, by checking that the proof applies to a general product.

By interpolation between the Strichartz part of the norm and the energy part of the $X^s(T)$-norm,
we have
\begin{align}
 \|\jb{\nb}^\eps v \|_{L^{q_1}_T L^{r_1}_x} 
\le \|v \|_{L^{q}_TL^r_x}^{1-\frac{\eps}s}   \| v \|_{L^\infty_T H^s_x}^{\frac{\eps}s}  
\leq \|v \|_{X^s(T)}
\label{J4}
\end{align}

\noi
for $0 <  \eps < s$, where
\[ \frac{1}{q_1} = \frac{1-\eps/s}{q} + \frac{\eps/s}\infty \qquad 
\text{and}\qquad \frac{1}{r_1} = \frac{1-\eps/s}{r} + \frac{\eps/s}2.\]

Similarly, 
by duality 
with $(L^1_T H^{s-1}_x+L^{\wt q}_T L^{\wt r}_x )^*
= L^\infty_T H^{1-s}_x \cap L^{\wt q'}_T L^{\wt r'}_x $
and interpolation, we have
\begin{align}
\| f \|_{L^1_T H^{s-1}_x+L^{\wt q}_T L^{\wt r}_x }
& = 
 \inf_{f = f_1+f_2} \Big(  
\| f_1 \|_{L^1_T H^{s-1}_x} +\| f_2 \|_{L^{\wt q}_T L^{\wt r}_x }\Big) \notag\\
& = \sup_{\|g\|_{L^\infty_T H^{1-s}_x \cap L^{\wt q'}_T L^{\wt r'}_x }\leq 1}
\bigg|\int_0^T \int_{\T^2} 
f g dx dt\bigg|\notag\\
& \le \sup_{\|g\|_{L^\infty_T H^{1-s}_x \cap L^{\wt q'}_T L^{\wt r'}_x }\leq 1}
 \| \jb{\nb}^{\eps} g \|_{L^{\wt q_2'}_T  L^{\wt r_2'}_x} 
 \| \jb{\nb}^{-\eps} f \|_{L^{\wt q_2}_T  L^{\wt r_2}_x} \notag\\
 & \les  \| \jb{\nb}^{-\eps} f \|_{L^{\wt q_2}_T  L^{\wt r_2}_x} 
\label{J5}
\end{align}

\noi
for $0 <  \eps < 1- s$, where
\[ \frac1{\wt q_2} = \frac{\eps/(1-s)}{1} + \frac{1-\eps/(1-s)}{\wt q}
 \qquad \text{and}\qquad \frac1{\wt r_2} = \frac{\eps/(1-s)}{2} + \frac{1-\eps/(1-s)}{\wt r}   .  \]

We also claim  that the following estimate holds:
\begin{equation}  
\big\| \jb{\nb}^{-\eps}  :\!\Psi^\l\!: v^{k-\l} \big\|_{L^{\wt q_2}_T L^{\wt r_2}_x }
 \les T^\theta \| \jb{\nb}^{-\eps} :\!\Psi^\l\!: \Vert_{L^\infty_{T, x}} 
 \|\jb{\nb}^{\eps} v \|^{k-\l}_{L^{q_1}_T L^{r_1}} .
\label{J6}
\end{equation}

\noi
for sufficiently small $\eps > 0$
and $\theta > 0$.
Fix $t \in [0, T]$.  By applying Lemma \ref{LEM:bilin} (ii), we have
\begin{align}
\big\| \jb{\nb}^{-\eps}  :\!\Psi^\l(t)\!: v^{k-\l}(t) \big\|_{ L^{\wt r_2}_x }
& \les \|\jb{\nb}^{-\eps} :\!\Psi^\l(t)\!: \|_{L^\frac{2}{\eps}_x} 
\| \jb{\nb}^{\eps} v^{k-\l} (t)\|_{L^{\wt r_2 }_x} \notag\\
& \le \|\jb{\nb}^{-\eps} :\!\Psi^\l(t)\!: \|_{L^\infty_x} 
\| \jb{\nb}^{\eps} v^{k-\l} (t)\|_{L^{\wt r_2 }_x} .
\label{J7}
\end{align}

\noi
Then, by applying Lemma \ref{LEM:bilin} (i), we have
\begin{align} 
\|\jb{\nb}^{\eps} v^{k-\l}(t)  \|_{L^{\wt r_2}_x}
 \les   \| v (t) \|_{L^{(k-\l)\wt r_2}_x}^{k-\l-1} \|\jb{\nb}^{\eps} v (t) \|_{L^{(k-\l) \wt r_2}_x}
 \les   \|\jb{\nb}^{\eps} v(t) \|^{k-\l}_{L^{(k-\l) \wt r_2}_x}.
\label{J8}
\end{align}

\noi
Note that we can choose $\eps > 0$ sufficiently small such that 
\begin{align}
  (k-1) \wt q_2 <  q_1 \qquad \text{and} \qquad (k-1) \wt r_2 \le r_1 
  \label{J9}
\end{align}

\noi
This can be achieved in view of \eqref{HK1} and 
\[ q_1 \to q, \quad  r_1 \to r, \quad \wt q_2 \to \wt q ,\quad \text{and} \quad  \wt r_2 \to \wt r \] 

\noi
as $\eps \to 0$.
Hence, \eqref{J6} follows from \eqref{J7}, \eqref{J8}, and \eqref{J9}. 
Note that the strict inequality in \eqref{J9} is used to gain a factor $T^\theta$.

Putting 
 Lemma \ref{LEM:Str},  \eqref{J4}, \eqref{J5}, and \eqref{J6} together, 
we obtain the desired estimate~\eqref{multilinear}.
\end{proof}

Proposition \ref{PROP:nonlin}
with a standard fixed point argument
immediately yields Theorem \ref{THM:1}
in the subcritical case, i.e.~$s > s_\text{crit}$.
In this case, we have $ q> k \wt q$, which  provides a positive power of $T$
on  the last terms of \eqref{J3a} and \eqref{J3b}. 
In particular, this implies that almost sure local well-posedness
holds on $[-T, T]$, where  $T = T_\o(\|\phi_0, \phi_1)\|_{\H^s})> 0$.
Note that 
the mild formulation \eqref{SNLW10}
with  the continuity of the linear propagator $S(t)$,
(the proof of) Proposition \ref{PROP:nonlin},
and Proposition \ref{PROP:sconv}
shows that the solution $v$ lies in $C([0, T]; H^s(\T^2))$.

On the other hand, in the critical case:
$s = s_\text{crit}$ (with  $k \geq 4$), 
 we have $ q= k \wt q$.
 Thus,   the last terms of \eqref{J3a} and \eqref{J3b}
 do not provide any power of $T$.
In this case, a direct application of Proposition \ref{PROP:nonlin}
 would yield only small data local well-posedness
 and thus we need to slightly  modify the argument .

Let $\eps > 0$ be as in Proposition \ref{PROP:nonlin}.
Then,  in view of \eqref{HK2} and \eqref{J4}, 
we set $Y^s(T)$ by 
\[ \| v\|_{Y^s(T)} = \max \big(\|v \|_{L^{q}_TL^r_x}^{1-\frac{\eps}s}   \| v \|_{L^\infty_T H^s_x}^{\frac{\eps}s}  , \,
\|v \|_{L^{q}_TL^r_x}\big).\]

\noi
Then, it follows from the proof of Proposition \ref{PROP:nonlin} that 
\begin{align} 
\| \G(v) \|_{Y^s(T)} 
& \les
\|S(t)(\phi_0, \phi_1)\|_{Y^s(T)}
+ \| :\!\Psi^k\!: \|_{L^1_TH^{s-1}_x} \notag\\
& \hphantom{XXX}+   T^\theta \sum_{\l=1}^{k-1} \|\jb{\nb}^{-\eps} :\!\Psi^\l\!: \|_{L^\infty_{T, x}} 
\| v \|_{Y^s(T)}^{k-\l}  
 +  T^{\frac{1}{\wt q}-\frac kq} 
\| v \|_{Y^s(T)}^k .  
\label{J10}
\end{align}

\noi
The difference estimate \eqref{J3b} with $X^s(T)$ replaced by $Y^s(T)$ also holds.

By the monotone convergence theorem, 
we have $\| v\|_{Y^s(T)}\to 0$  as $T \to 0$.
Hence, together with Proposition \ref{PROP:sconv}, 
we can choose $T = T_\o(\phi_0, \phi_1)> 0$ sufficiently small such that 
$\|S(t) (\phi_0, \phi_1)\|_{L^q_T  W^{\eps, r_1}_x}
+\| \!:\!\Psi^k \!: \!\|_{L^1_TH^{\s-1}_x} 
\leq \eta  \ll1 $ almost surely, allowing us to show that 
$\G$ is a contraction on the ball of radius $\eta$ in $Y^s(T)$.
Lastly, noting that \eqref{J10} holds even if we replace the $Y^s(T)$-norm on the left-hand side
by the $X^s(T)$-norm, 
we conclude  that $v \in X^s(T)$.

\section{Weak universality for semilinear wave equations with random
perturbation}
\label{SEC:4}

In this section, we present an application of the local well-posedness argument presented in Section \ref{SEC:3}.
In particular, 
we establish weak universality of the Wick ordered SNLW
in the following sense.
Given small $\eps > 0$, we consider the following  SNLW 
on $(\eps^{- 1} \T)^2$ 
with a smooth noise $\eta^\eps$:
\begin{align}
\begin{cases}
 \dt^2 w_\eps - \Dl w_\eps = f (w_\eps) + a (\eps, t) w_\eps + \delta   (\eps) \eta^\eps \\
 (w_\eps, \dt w_\eps) |_{t = 0}=  (0, 0),
\end{cases}
\qquad 
(x, t)\in (\eps^{- 1} \T)^2 \times \R_+, 
\label{SNLW11}
\end{align}

\noi
where 
 $f : \mathbb{R} \rightarrow
\mathbb{R}$ is a  smooth odd function which we take bounded with a sufficient
number of bounded derivatives\footnote{For example, in proving  Theorem \ref{THM:2}, 
it suffices to assume that $f(0) = f''(0) = 0$
and that there is a control up to the fourth derivative of $f$.
See also Remark \ref{REM:poly}.} 
and $a (\eps, t)$ and $\delta(\eps)$ are parameters we will fix below.
In the following, we take  the noise $\eta^\eps$ to be white in time but
smooth and stationary in space.\footnote{
Think of $\eta^\eps = \psi *_x \xi$ for some smooth function $\psi$ on $(\eps^{-1}\T)^2$.
Then, 
with $\wt \beta_n$ as in \eqref{G0a}, 
$\eta^\eps$ can be formally written  as 
\begin{align}
\eta^\eps(x, t) & 
=  \eps \sum_{n\in (\eps \Z)^{2}}
\ft \psi(n) d \wt \be_n (t) e^{2\pi i  n\cdot x} .
\label{MG0a}
\end{align}

}
 We point out that we could also work  with a Gaussian noise
$\eta^\eps$ which is regular both in space and time but, in order to fit more easily
in the general framework of this paper, we prefer to stick to a noise which is
white in time. 
Similarly,  we could work with a function $f$ with polynomial
growth.
For simplicity of the presentation, however, 
we work under the boudedness assumption on $f$.
 Indeed,  we will see that the precise form of $f$ does not matter in
the limit. 
See also Remark \ref{REM:poly} below.

Our aim is to describe the long time and large space behavior of
the solution $w_\eps$ to \eqref{SNLW11}.
 In order to do so,  we perform a change of variables
$u_{\eps} (x, t) \overset{\text{def}}{=} \eps^{- \g} w_\eps (\eps^{-1} x, \eps^{-1} t)$ 
and observe that $u_{\eps}$ satisfies
\begin{align}
\begin{cases}
 \dt^2 u_\eps - \Dl u_\eps
  = \eps^{- \gamma - 2} \big\{f
(\eps^{\gamma} u_{\eps}) + \eps^{\gamma} a
(\eps, \eps^{-1} t) u_{\eps}\big\} + \eps^{ - \g-\frac 12 } \delta
(\eps) \eta_{\eps}\\
(u_\eps, \dt u_\eps) |_{t = 0} = (0, 0),
\end{cases}
\label{SNLW12}
\end{align}

\noi
where $\eta_{\eps} (x, t) = \eps^{- \frac 3  2} \eta^\eps (\eps^{-1} x, \eps^{- 1} t)$. 
The normalization for $\eta_{\eps}$ has
been chosen in such a way that it converges as $\eps \rightarrow 0$
to a space-time white noise $\xi$ in law.
With this normalization, we  choose
$\delta (\eps) = \eps^{\g + \frac 1 2}$ in order for the
coefficient in front of $\eta_{\eps}$ to be $O_{\eps} (1)$ as
$\eps \rightarrow 0$.
For the sake of a simpler statement below, 
we apply Skorokhod's theorem\footnote{If we do not apply
Skorokhod's theorem here, then the conclusion of Theorem \ref{THM:2} 
holds only along some sequence $\{\eps_j\}_{j \in \N}$ tending to 0.}
and introduce a new noise with the same law, still denoted by $\eta_\eps$,
such that it converges to the whose noise $\xi$ 
almost surely.
We also use $u_\eps$ to denote the solution to \eqref{SNLW12}.
Then, letting $\Psi_\eps$ denote the  stochastic convolution $\Psi_\eps$ 
given by  $\Psi_{\eps} = (\dt^2 - \Dl)^{-1} \eta_\eps$, 
it follows from an argument analogous to the proof of Proposition \ref{PROP:sconv}
that $\Psi_\eps$ converges almost surely to the stochastic convolution $\Psi$ defined in \eqref{sconv1}
 in $C(\R_+; W^{\s, \infty}(\T^2))$ for any $\s <0$,
where we endow the space with the compact-open topology in time.

We now state the  main result of this section.

\begin{theorem}\label{THM:2}
  Let $\delta (\eps) = \eps^{\g + \frac 12}$ and $\gamma = 1$.
 Then,  there exists a choice of $a (\eps, t)$ such that, as $\eps \to 0$, 
 the family of 
the solutions   $\{u_{\eps}\}_{\eps>0}$ to \eqref{SNLW12} converges 
almost surely 
 to  the solution $u$ to the following  Wick ordered cubic SNLW:
  \begin{equation}
\dt^2 u - \Dl  u = \lambda : \!u^3 \!: \, + \, \xi
    \label{SNLW_ld}
  \end{equation}
  with zero initial data,
where the convergence takes place  
in $C([0, T_\o]; H^\s(\T^2))$, $\s < 0$,  for some  $T = T_\o(\Psi)>0$.
   Here the constant $\lambda = \lambda (f)$
  depends only on the function $f$.

\end{theorem}

\begin{proof}
In order to motivate the choice of $\gamma$, $a$, and the constant $\lambda$,
let us decompose $u_{\eps} = \Psi_{\eps} +v_{\eps}$ 
as in \eqref{decomp}.
Then, with our choice of $\dl(\eps) =  \eps^{\g+ \frac{1}{2}}$, 
we see that $v_{\eps}$ satisfies 
\begin{align}
 \dt^2 v_\eps - \Dl v_\eps = F_\eps(v_\eps), 
\label{SNLW13}
 \end{align}

\noi
where $F_\eps(v_\eps)$ is given by 
\[ F_{\eps} (v_{\eps}) \overset{\text{def}}{=} \eps^{-
   \gamma - 2} \big\{f (\eps^{\gamma} (\Psi_{\eps} +
   v_{\eps})) + \eps^{\gamma} a (\eps,\eps^{-1} t) 
   (\Psi_{\eps} + v_{\eps})\big\} . \]

Since $f$ is chosen to be odd, we have $f (0) = f'' (0) = 0$.
Then, Taylor's remainder theorem gives
\begin{align}
 F_{\eps} (v_{\eps}) = \eps^{- 2} \big\{f' (0) + a  (\eps, \eps^{-1} t)\big\} (\Psi_{\eps} + v_{\eps}) 
+ \eps^{2   \g - 2}  \frac{f^{(3)} (0)}{6}  (\Psi_{\eps} +
   v_{\eps})^3 + R_{\eps} 
\label{MG1}
\end{align}

\noi
with
\begin{align}
 R_{\eps} = \eps^{2 \gamma - 2} \int_0^1 \frac{(1 -
   \tau)^2}{2}  \big\{f^{(3)} (\tau \eps^{\gamma} (\Psi_{\eps} +
   v_{\eps})) - f^{(3)} (0)\big\}d \tau \cdot (\Psi_{\eps} +
   v_{\eps})^3 . 
\label{MG2}
\end{align}

\noi
From the explicit expression \eqref{H1a} for the Hermite polynomials, 
we have
\begin{align*}
  (\Psi_{\eps} + v_{\eps})^3 
  & = H_3 (\Psi_{\eps} +  v_{\eps}; \sigma_{\eps}) + 3 \sigma_{\eps}
   (\Psi_{\eps} + v_{\eps}), 
\end{align*}

\noi
where 
 $\s_{\eps} = \s_{\eps} (t) = \E
[\Psi_{\eps} (x, t)^2] \sim | \log \eps |$.\footnote{\label{foot12}
For simplicity, let  $\ft \psi(n) = \ind_{|n| \leq 1}$ in \eqref{MG0a}.
Then,  we have
\begin{align}
\eta_{\eps} (x, t) = \eps^{- \frac 3  2} \eta^\eps (\eps^{-1} x, \eps^{- 1} t)
=   \sum_{n\in  \Z^{2}}
\ind_{|n|\leq \eps^{-1}} \big(\eps^{-\frac12}d \wt \be_n (\eps^{-1}t)\big) e^{2\pi i  n\cdot x} 
\stackrel{d}{=}   \sum_{n\in  \Z^{2}}
\ind_{|n|\leq \eps^{-1}} d \wt \be_n (t) e^{2\pi i  n\cdot x}, 
\label{MG0}
\end{align}

\noi
where we use the white noise scaling in the last equality.
In view of \eqref{sig1} with \eqref{MG0}, it is easy to see 
the logarithmic divergence of $\s_\eps$ in this case.}
Hence, from \eqref{MG1} and \eqref{MG2}, we deduce that 
\begin{align*}
 F_{\eps} (v_{\eps}) = \eps^{- 2} \bigg\{ f' (0) + a (\eps, \eps^{-1} t) 
 & + 3 \eps^{2 \gamma} \sigma_{\eps}
\frac{f^{(3)} (0)}{6} \bigg\} (\Psi_{\eps} + v_{\eps}) \notag\\ 
& +\eps^{2 \gamma - 2}  \frac{f^{(3)} (0)}{6} H_3 (\Psi_{\eps} +v_{\eps}; \sigma_{\eps}) + R_{\eps}. 
\end{align*}

\noi
Therefore,  in order for $F_{\eps} (v_{\eps})$ to have a (non-trivial) 
finite limit (as a space-time distribution),  we must take
\[ \g = 1
   \qquad \text{and}\qquad 
   a (\eps, \eps^{-1} t) = - f' (0) - \eps^{2}
   \sigma_{\eps} (t) \frac{f^{(3)} (0)}{2}. \]

\noi
With these choices and letting 
$\ld =\frac{f^{(3)} (0)}{6} $,  we have
\begin{align*}
 F_{\eps} (v_{\eps}) 
= \lambda H_3 (\Psi_{\eps} +   v_{\eps}; \sigma_{\eps}) + R_\eps
 = \ld  :\! u_\eps^3\!: + R_\eps. 
\end{align*}

It remains to show that $R_\eps \to 0$ as $\eps \to 0$.
Let us analyze the behavior of $R_{\eps}$. Letting
\[ \Ld_\eps = \int_0^1 \frac{(1 - \tau)^2}{2}  \big\{f^{(3)} (\tau
   \eps^\g (\Psi_{\eps} + v_{\eps})) - f^{(3)} (0)\big\}
d \tau, \]

\noi
we have
\begin{align}
	 R_{\eps} = \Lambda_{\eps}  (\Psi_{\eps}^3 + 3
   \Psi_{\eps}^2 v_{\eps} + 3 \Psi_{\eps}
   v_{\eps}^2 + v_{\eps}^3). 
\label{MG4}
\end{align}

\noi
Moreover, by the fundamental theorem of calculus, we have
\[ \Ld_\eps = \eps^\g \int_0^1 \frac{(1 -   \tau)^2}{2} 
\int_0^\tau f^{(4)} (\al  \eps^{\gamma} (\Psi_{\eps} +   v_{\eps}))d\al d \tau 
\cdot (\Psi_{\eps} + v_{\eps}) \]

\noi
Thus, using the boundedness of the derivatives of $f$,  we have
\begin{align} 
| \Ld_\eps (x, t) | \les \eps^{\gamma} 
\big\{|    \Psi_{\eps} (x, t) | + | v_{\eps} (x, t) |\big\} . 
\label{MG5}
\end{align}

By Proposition \ref{PROP:sconv} and \eqref{MG0} in the footnote \ref{foot12}, 
 it is not difficult to see that
\begin{align}
 \eps^{\gamma} \| \Psi_{\eps} \|_{L^{\infty}_t([0, 1]; L^\infty_x)} =
 o_{\eps} (1) 
\label{MG6}
\end{align}

\noi
almost surely.
Hence, from \eqref{MG4}, \eqref{MG5}, and \eqref{MG6}, 
we conclude that 
\begin{align*}
 | R_{\eps} (x, t ) | \leq o_{\eps} (1) \big(1 + |   v_{\eps} (x, t) |\big)^4 . 
\end{align*}

\noi
In particular, we can write \eqref{SNLW13} as 
\begin{align*}
 \dt^2 v_\eps - \Dl v_\eps 
 & = \ld  :\! u_\eps^3\!: 
+ o_\eps\big(\jb{ v_{\eps}}^4\big) \notag\\
& = 
\ld   \sum_{\l=0}^3 {3\choose \l} :\! \Psi_\eps^\l \,  \!: v_\eps^{3-\ell}
  + o_\eps\big(\jb{ v_{\eps}}^4\big).
 \end{align*}

\noi
Then, by proceeding  as in Section \ref{SEC:3}
with a variant of Proposition \ref{PROP:nonlin} 
(with $k = 4$ in view of the fourth order error term), 
 we obtain an a priori bound on $v_\eps$, uniformly in $\eps > 0$.
Moreover, the local existence time $T = T_\o $ 
depends only on $\Psi$
and is independent of $\eps>0$.

Let $u$ be the solution to \eqref{SNLW_ld}.
In an analogous manner, we can estimate the difference $v -v_\eps$, 
where $v = u - \Psi$ as in \eqref{decomp}.
Together with the almost sure convergence of $\Psi_\eps$ to $\Psi$, 
we see that $u_\eps$ converges to $u$ in $C([0, T_\o]; H^\s(\T^2))$
for $\s < 0$.
\end{proof}

\begin{remark}\label{REM:poly}\rm
If $f$ is an odd polynomial of degree $M$, then 
we obtain the following bound on $\Ld_\eps$:
\begin{align*} 
| \Ld_\eps (x, t) | 
& \les 
\max\Big(\eps^{\g} \big\{|    \Psi_{\eps} (x, t) | + | v_{\eps} (x, t) |\big\} , 
(\eps^{\g} \big\{|    \Psi_{\eps} (x, t) | + | v_{\eps} (x, t) |\big\})^{(M-3)_+}\Big).
\end{align*}

\noi
Together with \eqref{MG4} and \eqref{MG6}, we obtain 
\begin{align*} | R_{\eps} (x, t ) | \leq o_{\eps} (1) 
\big(1 + |  v_{\eps} (x, t) |\big)^{\max(4, M)}. 
\end{align*}

 \noi
 Then, by applying 
 a variant of Proposition \ref{PROP:nonlin} 
(with $k = \max(4, M)$ in view of the  error term), 
 we obtain a uniform (in $\eps$) a priori bound on $v_\eps$
 and the convergence of $u_\eps$ to the solution $u$
 to \eqref{SNLW_ld} as above.

\end{remark}

\begin{remark}\rm
We can also consider  the following  SNLW 
on $(\eps^{- 1} \T)^2$: 
\[ \dt^2 w_\eps- \Dl w_\eps  = f (w_\eps) + a (\eps, t) w_\eps 
+ b   (\eps, t) w_\eps^3 + \delta (\eps) \eta^\eps \]
  with two parameters $a, b$ which can be ``tuned'' 
  so that,  via a similar procedure, 
  we can cancel  the cubic term in the asymptotics of the nonlinear term 
  and obtain the  quintic SNLW: 
    \begin{equation}
\dt^2 u - \Dl  u = \lambda : \!u^5 \!: \, + \, \xi
    \label{SNLW_ld2}
  \end{equation}

\noi
for some $\ld = \ld(f)$.
  In this case, by choosing $\g = \frac 12$, 
    the remainder  takes the  form
  \[ \wt{R}_{\eps} = \wt{\Lambda}_{\eps} 
     (\Psi_{\eps}^5 + 5 \Psi_{\eps}^4 v_{\eps} + 10
     \Psi_{\eps}^3 v_{\eps}^2 + 10 \Psi_{\eps}^2
     v_{\eps}^3 + 5 \Psi_{\eps} v_{\eps}^4 +
     v_{\eps}^5) \]
  with
  \[ \wt{\Lambda}_{\eps} = \int_0^1 \frac{(1 - \tau)^4}{4!} 
     \big\{f^{(5)} (\tau \eps^{\gamma} (\Psi_{\eps} +
     v_{\eps})) - f^{(5)} (0)\big\} d \tau \]

\noi
  which yields the analogous estimate
  \[| \wt \Ld_\eps (x, t) | \les \eps^{\gamma} 
\big\{|    \Psi_{\eps} (x, t) | + | v_{\eps} (x, t) |\big\} .   \]
 
 \noi
 This implies
\begin{align*} | \wt R_{\eps} (x, t ) | \leq o_{\eps} (1) \big(1 + |   v_{\eps} (x, t) |\big)^6 . 
\end{align*}

 \noi
 Then, by applying 
 a variant of Proposition \ref{PROP:nonlin} 
(with $k = 6$ in view of the sixth order error term), 
 we obtain a uniform (in $\eps$) a priori bound on $v_\eps$
 and the convergence of $u_\eps$ to the solution $u$
 to \eqref{SNLW_ld2}.
 One can similarly consider SNLW with more parameters to be tuned
 to obtained the septic Wick ordered SNLW, etc.

\end{remark}

\begin{ackno}\rm
T.O.~was supported by the ERC starting grant 
no.~637995 ``ProbDynDispEq''.
M.G.~and H.K.~were supported by the DFG through CRC 1060.

\end{ackno}

\end{document}